\documentclass[11pt]{amsart}
\usepackage{amsmath,amssymb}
\usepackage{color}
\usepackage[all,curve,frame]{xy}

\newtheorem{defi}{Definici\'{o}n}[section]
\newtheorem{coro}[defi]{Corollary}
\newtheorem{lem}[defi]{Lemma}
\newtheorem{rem}[defi]{Remark}
\newtheorem{prop}[defi]{Proposition}
\newtheorem{thm}[defi]{Theorem}
\newtheorem{ej}[defi]{Example}

\newcommand{\eps}{\varepsilon}

\newcommand{\benu}{\begin{enumerate}}
\newcommand{\enu}{\end{enumerate}}
\newcommand{\til}{\widetilde}
\newcommand{\al}{\alpha}

\newcommand{\ga}{\gamma}
\newcommand{\be}{\beta}

\newcommand{\fle}{\rightarrow}
\newtheorem{sinobservacion}[defi]{}
\newenvironment{sinob}{\begin{sinobservacion} \rm}{\end{sinobservacion} }

\begin{document}

\title[On cycles of length three]
{On cycles of length three}
\author[Chaio]{Claudia Chaio}
\address{Centro marplatense de Investigaciones Matem\'aticas, Facultad de Ciencias Exactas y
Naturales, Funes 3350, Universidad Nacional de Mar del Plata, 7600 Mar del
Plata, Argentina y CONICET Argentina}
\email{cchaio@mdp.edu.ar}

\author[Guazzelli]{Victoria Guazzelli}
\address{Centro marplatense de Investigaciones Matem\'aticas, Facultad de Ciencias Exactas y
Naturales, Funes 3350, Universidad Nacional de Mar del Plata, 7600 Mar del
Plata, Argentina}
\email{vguazzelli@mdp.edu.ar}

\author[Suarez]{Pamela Suarez}
\address{Centro marplatense de Investigaciones Matem\'aticas, Facultad de Ciencias Exactas y
Naturales, Funes 3350, Universidad Nacional de Mar del Plata, 7600 Mar del
Plata, Argentina}
\email{psuarez@mdp.edu.ar}

\maketitle

\begin{abstract} We prove that if $A$ is a string algebra then there are not three irreducible  morphisms between indecomposable $A$-modules such that its composition belongs to $\Re^{6} \backslash \Re^{7}$, whenever the compositions of two of them are not in  $\Re^{3}$. Moreover, for any positive integer $n \geq 3$, we show that there are $n$ irreducible morphisms such that their composition is in $\Re^{n+4} \backslash \Re^{n+5}$.
\end{abstract}

\section*{Introduction}
Introduced by Auslander and Reiten in the early 70's, the notion of irreducible morphisms plays an important role in the representation theory of artin algebras.

It is well-known that the composition of $n$ irreducible morphisms between indecomposable modules
over an artin algebra $A$ belongs to $\Re^n,$ the $n$-th power of the radical $\Re$ of the
module category. Such a composition could be a non-zero morphism in
$\Re^{n+1}$.  This is still  a problem of interest in the representation theory of artin algebras, and in the last years, there have been  some  advances in such a direction, see for example \cite{Cha},  \cite{CCT:07}, and
\cite{ChacotreII}.

In \cite{CCT:07}, Coelho, Trepode and the first named author characterized when the composition of two irreducible morphisms is non-zero and belongs to $\Re^{3}$. Moreover, they proved that if two irreducible morphisms between indecomposable $A$-modules such that their composition is non-zero and belongs to a greater power of the radical, greater than two, then such composition  is at least in $\Re^{4}$.

Later in \cite{AC}, Alvares and Coelho proved that if $f$ and $g$ are irreducible morphisms  between  indecomposable $A$-modules such that $0 \neq fg \in \Re^{3}$ then $fg \in \Re^{5}$. Furthermore, they showed
an example of two irreducible morphisms whose composition is in $\Re^{5} \backslash \Re^{6}$. To prove such a result they used a result due to Hoshino, proved in \cite{H}, that if a module $X$ in $\Gamma_A$ is such that $DTrX=X$, then either the connected component of $\Gamma_A$ which contains $X$ is a homogeneous stable tube or $A$ is a local Nakayama algebra.

Finally, in \cite{Cha2}, the first named author generalized the result proven in \cite{AC}.
Precisely,  the author  proved that given an artin algebra $A$ where the configurations of almost split sequences have at most two indecomposable middle terms,
then the  non-zero composition of $n$ irreducible morphisms on a left almost pre-sectional path is such that it belongs to $\Re^{n+3}$ for $n \geq 1$.

As a consequence of the above mentioned result, for any artin algebra, we know that if  the non-zero composition of any three irreducible morphisms ${h_{i}}$ between indecomposable $A$-modules,  is such that $h_{3}h_2 h_{1} \in \Re^{4}$, $h_{3}h_2 \notin \Re^{3}$ and  $h_2 h_{1} \notin \Re^{3}$ then $h_{3}h_2 h_{1} \in \Re^{6}$.

A natural question now is if the composition of three irreducible morphisms between indecomposable $A$-modules can be in $\Re^{6} \backslash  \Re^{7}$, whenever the composition of any  two of them  are not in $\Re^{3}$, that is, behaves well.

In this work, we prove that if $A$ is a string algebra then there are not three irreducible morphisms such that their composition is in $\Re^{6} \backslash  \Re^{7}$, if the composition of any two of them is not in $\Re^{3}$. Furthermore, for a string  algebras we prove that the minimum for three irreducible morphisms in such a condition is seven.

We also find families of algebras where their module category have $n$ irreducible  morphisms between indecomposable modules such that its composition belongs to $\Re^{n+4} \backslash \Re^{n+5}$ for $n \geq 3$, whenever the compositions of $n-1$ of them belong to $\Re^{n-1} \backslash \Re^{n}$.
It is still an open problem to see if the minimum $n$ is equal to $n+3$, for $n \geq 3$.
\vspace{.1in}

The paper is organized as follows. The first section is dedicated to recall some preliminaries definitions and results. In section 2, we prove  some general results concerning algebras which have cycles of length three. In section 3, we present some strings algebras that contains irreducible morphism from $M$ to $\tau M$, for $M$ an indecomposable $A$-module. In Section 4, we prove  some technical lemmas and apply the results of the previous sections to prove that if we consider a string algebra there are not three irreducible morphisms such that their composition is in $\Re^{6} \backslash \Re^{7}$ whenever the composition of any two of them behaves well. Finally, in the last section we give families of algebras having $n$ irreducible morphisms such that their composition belongs to $\Re^{n+4} \backslash \Re^{n+5}$, for $n\geq 3$ and such that the composition of $n-1$ of them behaves well.
\vspace{.1in}

\thanks{The authors thankfully acknowledge partial support from CONICET
and from Universidad Nacional de Mar del Plata, Argentina. The problem solved  in this article come up from the question of "Which algebras have an irreducible morphism from $M$ to $\tau M$" asked  by E. R. Alvares to the second named author, when she was visiting Universidade Federal do Paran\'a in Curitiba. This question is still an open problem.
The first author is a researcher from CONICET.}

\section{preliminaries}

\begin{sinob} A {\bf quiver} $Q$ is given by a set of vertices $Q_0$ and
a set of arrows $Q_1$, together with two
maps $s,e:Q_1\fle Q_0$. Given an arrow $\al\in Q_1$, we write $s(\al)$
the starting vertex of $\al$ and $e(\al)$
the ending vertex of $\al$.
For each arrow $\al\in Q_1$ we denote by $\al^{-1}$ its formal inverse, where $s(\al^{-1})=e(\al)$ and
$e(\al^{-1})=s(\al)$.

A {\bf walk} in $Q$  is a concatenation  $c_1\dots c_n$, with $n\geq1$, such that $c_i$ is either an
arrow or the inverse of an arrow, and $e(c_i)=s(c_{i+1})$.
We say that $c_1\dots c_n$ is a {\bf reduced walk} provided $c_{i}\neq c_{i+1}^{-1}$ for each $i$,
$1\leq i\leq n-1$.

If  $A$ is an algebra then
there exists a quiver $Q_A$, called the {\bf ordinary quiver of}  $A$, such that
$A$ is the quotient of the path algebra $kQ_A$ by an admissible ideal.
\end{sinob}

\begin{sinob}
Let $A$ be an  artin algebra. We denote by $\mbox{mod}\,A$ the category of finitely generated
left $A$-modules and  by $\mbox{ind}\,A$ the full subcategory of $\mbox{mod}\,A$ which consists of
one representative of each isomorphism class of indecomposable $A$-modules.

Let $X$ be a non-projective (non-injective) indecomposable $A$-module.
By $\al(X)$ ($\al'(X)$, respectively) we denote the number of indecomposable summands in the
middle term of an almost split sequence ending (starting, respectively) at $X$.
We say that $\al(\Gamma)\leq 2$ if $\al(X)$ and $\al'(X)$ are less than or equal to two, whenever they are defined.
\end{sinob}

\begin{sinob}
A morphism $f :X  \rightarrow  Y$, with $X,Y \in \mbox{ mod}\,A$,
is called {\bf irreducible} provided it does not split
and whenever $f = gh$, then either $h$ is a split monomorphism or $g$ is a
split epimorphism.

If $X,Y \in \mbox{mod}\,A$, the ideal $\Re(X,Y)$ is
the set of all the morphisms $f: X \rightarrow Y$ such that, for each
$M \in \mbox{ind}\,A$, each $h:M \rightarrow X$ and each $h^{\prime }:Y\rightarrow M$
the composition $h^{\prime }fh$ is not an isomorphism. For $n \geq 2$, the powers of $\Re(X,Y)$
are defined inductively. By $\Re^\infty(X,Y)$ we denote the intersection
of all powers $\Re^i(X,Y)$ of $\Re(X,Y)$, with $i \geq 1$. \vspace{.05in}

By \cite{B}, it is well-known that a morphism $f :X  \rightarrow  Y$, with $X,Y \in \mbox{ind}\,A$,
is irreducible if and only if $f \in \Re(X,Y)\setminus \Re^2(X,Y)$.\vspace{.05in}

We recall the definition of degree of an irreducible morphism given by S. Liu in \cite{L}.%\vspace{.05in}

Let $f:X\rightarrow Y$ be an irreducible morphism in
$\mbox{mod}\,A$, with $X$ or $Y$ indecomposable. The {\bf left degree} $d_l(f)$ of $f$ is infinite,
if for each integer $n\geq 1 $, each module $Z\in \mbox{ind}\,A$
and each morphism $g:Z  \rightarrow X$ with $g \in \Re^{n}(Z,X) \backslash \Re^{n+1}(Z,X)$ we
have that $fg \notin \Re^{n+2}(Z,Y)$. Otherwise, the left degree of
$f$ is the least natural number $m$ such that there is an $A$-module $Z$
and a morphism $g:Z  \rightarrow X$ with $g \in \Re^{m}(Z,X) \backslash \Re^{m+1}(Z,X)$ such that
$ fg \in \Re^{m+2}(Z,Y)$.

The {\bf right degree} $d_r(f)$ of an irreducible morphism $f$
is dually defined.
\vspace{.05in}

We denote by $\Gamma_A$ its Auslander-Reiten
quiver, by $\tau$ the Auslander-Reiten translation, and $\tau^{-1}$ its inverse.

Let $X\rightarrow Y$ be an arrow in $\Gamma_A$. Assume that $f: X\rightarrow Y$ is an
irreducible morphism in $\mbox{mod} \,A$. Following \cite{L}, we define the left degree of the arrow $X\rightarrow Y$ to be
$d_l(f)$, and the right degree of the arrow $X\rightarrow Y$ to be $d_r(f)$.

\begin{lem} \label{liu-cha} Let $A$ be a finite dimensional $k$-algebra. Any cycle of irreducible morphisms between indecomposable $A$-modules has both a monomorphism and an epimorphism.
\end{lem}
\begin{proof}
By \cite[Lemma 2.2]{L}, we know that every oriented cycle in $\Gamma_A$  contains both an arrow of finite left
degree and an arrow of finite right degree.

By \cite[Corollary 3.2]{CLT}, the arrows of finite left
degree and the ones of finite right degree correspond to irreducible epimorphisms and monomorphisms, respectively. Then we get the result.
\end{proof}

An indecomposable $A$-module $M$ is \textbf{left (right) $\tau$-stable}
if for all positive integer $n$ the module $\tau ^n M$ ($\tau ^{-n} M$) is defined.
An indecomposable $A$-module $M$ is \textbf{$\tau$-stable} if it is both left and right $\tau$-stable.

In particular, if a $\tau$-stable module  $M$ satisfy that $\tau^m M\simeq M$
for some positive integer $m$, then we say that $M$ is \textbf{$\tau$-periodic}.
Moreover, $M$ is $\tau$-periodic of rank $m$ if $\tau^m M\simeq M$ and $\tau^k M \ncong M$ for ll $1\leq k<m$.

A path $M_1\fle M_2\fle \dots \fle M_n$ of irreducible morphisms
with $M_j\in \mbox{ind}\,A$ for $j=1,...,n$ and $n\geq 3$ is called { \bf sectional}
if for each $j=3,...,n$ we have that $M_{j-2}\not \simeq \tau M_j$.

A path $Y_0\rightarrow Y_1\rightarrow \dots \rightarrow Y_n$ in
$\Gamma _{A}$ is {\bf presectional} if for each $i$, with $1\leq
i\leq n-1,$ such that $Y_{i-1} \simeq \tau Y_{i+1}$ then there is an irreducible  morphism
$Y_{i-1}\oplus \tau Y_{i+1}\rightarrow Y_i$. Equivalently, if $\tau ^{-1}Y_{i-1}\simeq Y_{i+1},$ then there is an irreducible  morphism
$Y_i\rightarrow \tau ^{-1}Y_{i-1}\oplus Y_{i+1}$. Note that a sectional path  is also presectional.

A path $Y_0\rightarrow Y_1\rightarrow  \dots \rightarrow Y_n$ in
$\Gamma _{A}$ is {\bf left almost presectional} if $Y_0\rightarrow Y_1\rightarrow  \dots \rightarrow Y_{n-1}$
is presectional in $\Gamma _{A}$ and $Y_{n}\simeq \tau^{-1}Y_{n-2}$. Dually, we can define a right almost presectional path.

In \cite{Cha2}, the first named author gave a generalization of the result proven in \cite{AC}.
Moreover, as a consequence of such result  the author got Corollary \ref{tres}.

\begin{thm} \label{composn}
Let $A$ be an artin algebra and
assume that there is a configuration of almost split sequences as follows

$$
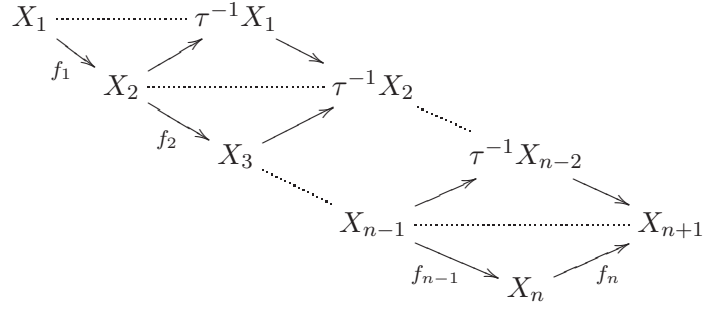
\xymatrix  @R=0.3cm  @C=0.5cm {
	X_1\ar[dr]_{f_1}\ar@{.}[rr]&& \tau^{-1}X_1\ar[dr]&&&\\
&X_2\ar[dr]_{f_2}\ar[ur]\ar@{.}[rr]&& \tau^{-1}X_2\ar@{.}[dr]&&\\
&& X_3\ar[ur]\ar@{.}[dr]&&\tau^{-1}X_{n-2}\ar[dr]&\\
&&&X_{n-1}\ar[dr]_{f_{n-1}}\ar[ur]\ar@{.}[rr]&& X_{n+1}\\
&&&&X_{n}\ar[ur]_{f_n}&}$$
\vspace{.05in}

\noindent where  $f_1:X_1 \rightarrow X_2, \dots, f_n: X_n \rightarrow X_{n+1}$ are
irreducible morphisms between indecomposable $A$-modules  with $f_{1} \dots f_{n-1}$ in a
left almost pre-sectional path such that $f_{n-1} \dots f_1 \notin \Re^{n}$.
Let $h_i:X_i \rightarrow X_{i+1}$ be irreducible morphisms for $i=1, \dots, n$ such that $0 \neq h_n...h_1\in \Re^{n+1}$.
Then, $h_n...h_1\in \Re^{n+3}$.
\end{thm}

\begin{coro} \label{tres} Let $A$ be an artin algebra and ${h_{i}}:X_{i}{\rightarrow }X_{i+1}$ be
irreducible morphisms with $X_i \in \mbox{ind}\,A$ for $i=1, 2 ,3$ such that $h_{3}h_2 h_{1} \in \Re^{4}(X_{1}, X_{4})$. Then, $h_{3}h_2 h_{1} \in \Re^{6}(X_{1}, X_{4})$.
\end{coro}
\end{sinob}

\begin{sinob}
Let $A$ be an algebra such that $A\cong kQ_A/I_A$. The algebra $A$ is called a {\bf string algebra} provided:

\begin{enumerate}
\item[(1)] Any vertex of $Q_A$ is the starting point of at most two arrows.
\item[(1')] Any vertex of $Q_A$ is the ending point of at most two arrows.
\item[(2)] Given an arrow $\beta$, there is at most one arrow $\gamma$ with $s(\beta)=e(\gamma)$ and $\gamma \beta\notin I_A.$
\item[(2')] Given an arrow $\gamma$, there is at most one arrow $\beta$ with $s(\beta)=e(\gamma)$ and $\gamma \beta \notin I_A.$
\item[(3)] The ideal $I_A$ is generated by a set of paths of $Q_A$.
\end{enumerate}
\vspace{.05in}

Let $A=kQ_A\diagup I_A$ be a string algebra. A {\bf string} in $Q_A$ is either a trivial path
$\varepsilon_v$ with $v\in Q_0$, or a reduced walk $C=c_1\dots c_n$ of length $n\geq 1$ such that
no sub-walk $c_{i}\dots c_{i+t}$ nor its inverse belongs to $I_A$.
We say that a string $C=c_1\dots c_n$ is {\bf direct (\bf inverse)}
provided all $c_i$ are arrows (inverse of arrows, respectively). We
consider the trivial walk $\eps_v$ a direct as well as an inverse string.

We say that a string $C$ has length $n$ if the number of arrows and inverse of arrows in its composition is $n$.

For each string $C=c_1\dots c_n$ in $Q_A$, an indecomposable string $A$-module $M(C)$ is defined.
Conversely, given $M$ an indecomposable string
$A$-module there exists a "unique" string $C$ such that $M=M(C)=M(C^{-1})$. The band modules are defined over strings $C$ such that all powers $C^n$, with $n\in \mathbb{N}$ are defined, see \cite{BR}. Every module over a string algebra
is defined either as a string module or as a band module, see \cite{BR}.
Moreover, if $A$ is a representation-finite string algebra then all the indecomposable $A$-modules are strings ones.

We say that a string $C$ {\bf starts in a deep (\bf on a peak)} provided there is no arrow $\beta$
such that $\beta^{-1}C$ ($ \beta C$, respectively) is a string. Dually, a string $C$ {\bf ends in a deep
(on a peak)} provided there is no arrow $\beta$
such that $ C\beta $ ($ C\beta^{-1}$, respectively) is a string.\vspace{.05in}

By \cite{BR} we know  that given a string algebra $A$ then
$\al(\Gamma)\leq 2$. Moreover, the authors also described all the almost split sequences of $\mbox{mod}\,A$ in terms of strings.

Consider  $I(u)$ to be the injective module corresponding to the vertex $u \in (Q_A)_0$. Then, $I(u)=M(D_1D_2)$
where $D_1$ is a direct string starting on a peak and $D_2$ is an inverse string ending on a peak.

Dually, if $P(u)$ is the projective corresponding to $u \in Q_0$ then $P(u)=M(C_1C_2)$ where $C_1$ is an inverse string and $C_2$ is a direct string. Moreover, $C_1C_2$ is a string that starts and ends in a deep.

For a detail account on these algebras see \cite{BR} and for general Auslander-Reiten theory we refer the reader to \cite{ASS} and \cite{ARS}.
\end{sinob}

% \section{On the composition of $n$ irreducible morphisms in $\Re^{n+3}\backslash \Re^{n+4}$}

\section{ General results}

Consider the following family of quivers $Q_{n}$

\begin{displaymath}
\xymatrix  @R=0.3cm  @C=0.6cm {
	1 \ar@(ul,dl)[]_{\al}\ar[r]^{\beta_1}&2\ar[r]^{\beta_2}&\dots \ar[r]^{\beta_n}& n+1&&
	}
\end{displaymath}
\vspace{.05in}

\noindent for $n\geq 2$ and  the ideal $I=<\al^2,\be_1\be_2>$. We denote the algebras $kQ_{n}/I$  by $(W(n),I).$
\vspace{.05in}

Fix an integer $n\geq 3$, and consider any algebra $A \simeq (W(n),I)$.  In such algebras  there is a composition of $n$ irreducible morphisms $h_i:X_i \rightarrow X_{i+1}$ for $i=1, \dots,  n$ between indecomposable $A$-modules such that $h_n \dots h_1\in \Re^{n+3}(X_1,X_{n+1})\backslash\Re^{n+4}(X_1,X_{n+1})$, with $h_n \dots h_2 \in \Re^{n}(X_2,X_{n+1})$.
\vspace{.1in}

We illustrate the above situation in the next example.

\begin{ej}\label{ej3.2.1MtauM}
\emph{Consider the algebra $A\simeq (W(3),I)$.
The Auslander-Reiten quiver $\Gamma_{A}$ is the following:}

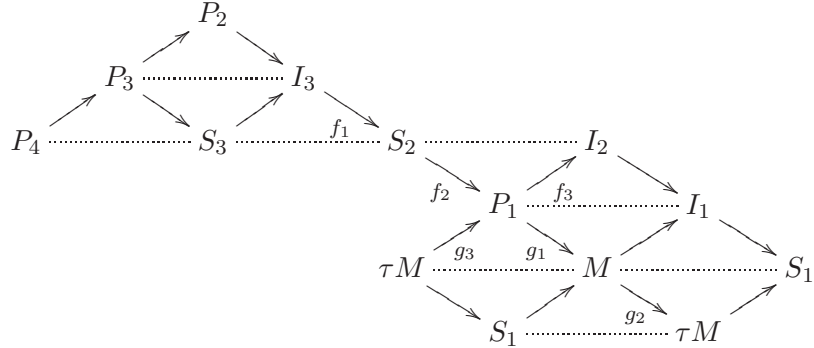
\begin{displaymath}
\xymatrix  @R=0.3cm  @C=0.6cm {
	&&P_2 \ar[dr]&&&&&&\\
&P_3\ar@{.}[rr]\ar[ur]\ar[rd]&& I_3\ar[rd]_{f_1}&&&&&\\
P_4\ar[ur]\ar@{.}[rr]&&S_3\ar[ur]\ar@{.}[rr]&&S_2\ar@{.}[rr]\ar[dr]_{f_2}&&I_2\ar[rd]&&\\
&&&&&P_1\ar@{.}[rr]\ar[ur]_{f_3}\ar[rd]_{g_1}&& I_1\ar[rd]&\\
&&&&\tau M\ar@{.}[rr]\ar[ur]_{g_3}\ar[rd]&&M \ar@{.}[rr]\ar[ur]\ar[rd]_{g_2}&&S_1\\
&&&&&S_1\ar@{.}[rr]\ar[ur]&&\tau M\ar[ur]& }
\end{displaymath}

\noindent \emph{where we identify the modules which are the same.}

\emph{Consider the irreducible morphisms $f_1:I_3 \rightarrow S_2$, $f_2:S_2 \rightarrow P_1$ and $f_3:P_1 \rightarrow I_2$.}

\emph{We define  $h_2:S_2 \rightarrow P_1$ as follows $h_2=f_2+g_3g_2g_1f_2$, where $g_1: P_1 \rightarrow M$, $g_2: M \rightarrow \tau M$ and $g_3: \tau M \rightarrow P_1$ are irreducible morphisms. Then $h_2$  is irreducible. Indeed, otherwise,  $h_2 \in \Re^{2}(S_2, P_1)$. Therefore, $f_2 \in \Re^{2}(S_2, P_1)$ a contradiction since $f_2$ is an irreducible morphism between indecomposable modules. Note that the composition $f_3h_2 f_1\in \Re^6(P_2,I_2)\backslash \Re^7(P_2,I_2)$, but the composition  $f_3h_2\in \Re^{3}(I_3,I_2)$.}
\end{ej}

We are interested in finding three irreducible morphisms between indecomposable $A$-modules such that their composition
belongs to $\Re^{6}\backslash \Re^{7}$,
and moreover, with the property that the composition  of two of such morphisms does not belong to $\Re^3$.

In Section 4,  we shall prove that if $A$ is a string algebra then there are not irreducible morphisms $h_i$ for $i=1, 2, 3$ between indecomposable $A$-modules in  $\Re^6\backslash \Re^7$, with $h_2 h_1 \notin \Re^3$ and  $h_3 h_2 \notin \Re^3$.

Throughout this paper, we shall prove all our results for the composition of three irreducible morphisms $h_i$ for $i=1, 2$ and $3$, such that $d_l(h_3)=2$. We observe that, with similar arguments one can prove the results for the case where  $d_r(h_3)=2$.
\vspace{.1in}

Now, we show that if for some artin algebra $A$, there are morphisms as described above, then there must be a cycle of irreducible morphisms between indecomposable $A$-modules of length three.

\begin{prop}\label{cycle3}
Let $A$ be an artin algebra and let $f_1:X \rightarrow Y,\, f_2:Y \rightarrow W,$ and $f_3:W \rightarrow V$
be irreducible morphisms between   indecomposable $A$-modules such that $f_3f_2f_1\in \Re^6(X,V)\backslash \Re^7(X,V)$
with $f_2f_1\notin \Re^3(X,W)$ and  $f_3f_2\notin \Re^3(Y,V)$. Then, there exists a cycle
of length three.
\end{prop}

\begin{proof}
Since $f_3f_2f_1\in \Re^6(X,V)\backslash \Re^7(X,V)$, then there is a path $\psi$ of irreducible morphisms

$$\xymatrix  @R=0.3cm  @C=0.6cm {
\psi:X\ar[r]^{g_1}&A_1\ar[r]^{g_2}&A_2\ar[r]^{g_3}&A_3\ar[r]^{g_4}&A_4\ar[r]^{g_5}&A_5\ar[r]^{g_6}&V
}$$

\noindent such that $\psi\notin \Re^7(X,V)$. Moreover, since $0  \neq f_3f_2f_1\in \Re^4(X,V)$, $f_2f_1\notin \Re^3(X,W)$ and  $f_3f_2\notin \Re^3(Y,V)$ then by \cite[Theorem 2.2]{CCT}
there is a configuration of almost split sequences as follows:

\begin{equation} \label{tres-morf}
\xymatrix  @R=0.3cm  @C=0.6cm {
	X\ar[dr]_{h_1}\ar@{.}[rr]&& Z\ar[dr]&&\\
&Y\ar[dr]_{h_2}\ar[ur]_{h_4}\ar@{.}[rr]&& V&&\\
&& W\ar[ur]_{h_3}&&&&}
\end{equation}

\noindent such that $h_3h_2h_1=0$, $\alpha'(X)=1$ and $\alpha'(Y)=2$ or its dual.

By \cite[Lemma 2.3]{ChacotreII} and the fact that $d_l({h_4})< \infty$, the dimension of the irreducible morphisms involved in $(\ref{tres-morf})$ is one.
Since $\alpha'(X)=1$ and $g_1:X\fle A_1$ is irreducible then $A_1\simeq Y$.

We claim that $A_2\simeq W$. In fact, if $A_2\simeq Z$ then $g_1= \alpha_1h_4+\mu_1$ and $g_2=\alpha_2h_4+\mu_2$ with $\alpha_1,\alpha_2\in k^{*}$ and $\mu_1,\mu_2\in \Re^2$. Since $h_4h_1=0$ we have that
$g_2g_1=\alpha_2h_4\mu_1+\alpha_1\mu_2h_1+\mu_2\mu_1\in\Re^3(X,Z)$.  Therefore, we get that $\psi\in \Re^7(X,V)$  a contradiction to our assumption. This establishes our claim.

With similar arguments as above we can prove that $A_3\not \simeq V$.

On the other hand, since  $\alpha(V)=2$ and there are irreducible morphisms $A_5 \rightarrow  V$, $Z \rightarrow V$ and $W \rightarrow V$, then
$A_5\simeq Z$ or $A_5\simeq W$.
If $A_5\simeq W$ is easy to see that there is a cycle $W\fle A_3\fle A_4\fle W$
 of length three.
Now, if $A_5\simeq Z$, since $\alpha(Z)=1$ then $A_4\simeq Y$. Hence, the path $\psi$ is as follows:

$$\xymatrix  @R=0.3cm  @C=0.6cm {
\psi:X\ar[r]^{g_1}&Y\ar[r]^{g_2}&W\ar[r]^{g_3}&A_3\ar[r]^{g_4}&Y\ar[r]^{g_5}&Z\ar[r]^{g_6}&V.
}$$

\noindent Then, there is a cycle $Y\fle W\fle A_3\fle Y$  in $\mbox{mod}\,A$ of length three.
\end{proof}

Next, we present a characterization for the existence of cycles of length three in $\mbox{mod}\,A$.

\begin{thm}\label{MtauM}
Let $A$ be an artin algebra. The following conditions are equivalent.
\benu
\item[(a)] There is a cycle in $\emph{mod}\,A$ which is a composition of irreducible morphisms between indecomposable $A$-modules of length three.
\item[(b)] There is an indecomposable not projective $A$-module $M$ and an irreducible morphism from $M$ to $\tau \,M$.
\enu
\end{thm}

\begin{proof}
$(a) \Rightarrow (b)$. By hypothesis there is a cycle of irreducible morphisms between indecomposable $A$-modules of length three. Let $M \rightarrow  M_1 \rightarrow M_2 \rightarrow M$ be such a cycle.
By \cite[Theorem 7]{BS},  any  path of the form $M \rightarrow  M_1 \rightarrow M_2 \rightarrow M \rightarrow M_1$
is not sectional. Therefore, one of the following conditions hold.
\benu
\item[(1)] $M\simeq \tau M_2$;
\item[(2)] $M_1\simeq \tau M$; or
\item[(3)] $M_2\simeq \tau M_1$.
\enu

In the former case, there is an irreducible morphism from $M_2$ to $\tau M_2$. In case (2) there is an irreducible morphism from  $M$ to $\tau M$. Finally, in the latter case, we have an irreducible morphism  $M_1$ to $\tau M_1$.

In conclusion, in all the cases, there is an indecomposable $A$-module which is not projective and an irreducible morphism from that module to the Auslander-Reiten translate of such a module, proving (b).

$(b) \Rightarrow (a)$. Let  $M$ be a module as in Statement (b). First, suppose that $M$ is not injective. Then $\tau^{-1}M$ is defined and there is an
irreducible morphism from $\tau M$ to $\tau^{-1}M$. Moreover, there is an irreducible morphism from $\tau^{-1}M$ to $M$. Hence there is a path of irreducible morphisms between indecomposable $A$-modules

$$\tau^{-1}M \fle M\fle \tau M \fle \tau^{-1}M$$

\noindent which is a cycle in $\mbox{mod}\,A$ of length three.

Secondly, if  $M$ is injective, then $\tau M$ is not projective. In fact, otherwise, we get to the contradiction that
the irreducible morphism from $M$ to $\tau M$ is both a monomorphism and an epimorphism. Hence, $\tau^2M$ is defined.

With similar arguments as before, there is an irreducible morphism from  $\tau^2M$ to $M$ and an
irreducible morphism  from $\tau M$ to $\tau^2M$. Therefore, there is a path of  irreducible morphisms

$$M \fle \tau M\fle \tau^{2}M \fle M$$

\noindent which is clearly a cycle of length three, getting (a).
\end{proof}

\begin{rem} \emph{We observe that for any positive integer $n$, condition (b) state below implies condition (a).}
\benu
\item[(a)] \emph { There is a cycle in $\mbox{mod}\,A$ which is a composition of irreducible morphisms between indecomposable $A$-modules of length $2n+1$.}
\item[(b)] \emph { There are indecomposable not projective $A$-modules $\tau^{i}M$ for $i=1, \dots, n-1$ and an irreducible morphism from $M$ to $\tau^{n} \,M$.}
\enu
\end{rem}

\begin{prop}
Let $A$ be an artin algebra.
Consider an indecomposable $A$-module  $M$ such that there is an irreducible morphism from $M$ to $\tau \,M$. If $M$ is
$\tau$-stable, then $M$ is $\tau$-periodic of rank three.
\end{prop}

\begin{proof}
Consider $M$ an indecomposable $\tau$-stable  $A$-module such that there is an irreducible morphism from $M$ to $\tau M$. Assume that  $M$
is not $\tau$-periodic. Then for all integer $n$, the modules $\tau^n M$ are defined. Moreover, for all integers  $r$ and $s$ such that $r\neq s$, then
$\tau^s M \not \simeq  \tau^r M$.

Since there is an irreducible morphism from $M$ to $\tau M$, then there is an irreducible morphism from $\tau^{k}M$ to $\tau^{k+1}M$ for every integer $k$.
Furthermore, there is an irreducible morphism from $\tau^2M$ to $M$. Hence, for  all integer $k$
there is an irreducible morphism from $\tau^{k}M$ to $\tau^{k-2}M$.

Consider a full subquiver  $\Gamma$ of $\Gamma_{A}$
consisting of modules of the form $\tau^k M$ for  all integer $k$.
Observe  that all the modules in $\Gamma$ are neither projective nor injective. Then for all module $\tau^k M$ in $\Gamma$,
we have that  the morphism $\tau^kM \rightarrow \tau^{k+1}M \oplus \tau^{k-2}M$ is irreducible. Since $\tau^{k+1}M \not \simeq \tau^{k-2}M$, then all the almost split sequences in $\Gamma$ have at least two indecomposable middle terms.
By Theorem \cite[Teorema 2.3]{L} there are not oriented cycles in $\Gamma$, a contradiction to Theorem \ref{MtauM}. Then $M$ is $\tau$-periodic.

We claim that $M$ has $\tau$-period three. In fact, let $n$ be the $\tau$-period of $M$, that is, $M \simeq \tau^{n}M$ and $M\not \simeq \tau^kM$ for $1\leq k<n$.
Since there are irreducible morphisms from $M $ to $\tau M$ and from $\tau^2M$ to $M$
and there are not loops in $\Gamma_{A}$, then $n>2$.

On the other hand,
since there is an irreducible morphism  $M$ to $\tau M$
there is a cycle in $\Gamma_{A}$ of the form

$$\psi: M\fle \tau M\fle \tau^2M \fle \dots \fle \tau^{n-1} M \fle \tau^{n} M \simeq M.$$

\noindent By \cite[Theorem 7]{BS}, we know that the path $M\stackrel{\psi}\rightsquigarrow M\fle \tau M$ is not sectional. Then $\tau^{k}M\simeq \tau(\tau^{k+2}M)\simeq \tau^{k+3}M$ for some $k\leq n$.
In conclusion, for any $k$ satisfying the above condition, we have that
$M\simeq \tau^3M$, proving the result.
\end{proof}

\begin{rem} \emph{In case that  $M$ is an indecomposable $\tau$-stable $A$-module such that there is an irreducible morphism from $M$ to $\tau^{n} \,M$ for $n \geq 3$, then $M$ is $\tau$-periodic of rank $n+3$.}
\end{rem}

\section{On some string Algebras}

We shall present some string algebras such that their module category has an irreducible morphism from $M$ to $\tau\,M$, with $M$ an indecomposable module. This results shall be fundamental  to prove that if we consider a string algebra $A$ then there are not three irreducible morphisms between indecomposable $A$-modules in $\Re^{6} \backslash \Re^{7}$, when the composition of two of then behaves well.
\vspace{.1in}

We start given a characterization of the string algebras which have an irreducible morphism from $M$ to $\tau\,M$,  where $M$ is an indecomposable not $\tau$-stable module.

\begin{prop} \label{not}
Let $A=kQ_A/I_A$ be a string algebra. The following conditions are equivalent.
\begin{enumerate}
  \item There is an indecomposable not $\tau$-stable $A$-module $M$, and an irreducible morphism from $M$ to $\tau\,M$.
  \item The quiver $Q_A$ has one of the following full subquivers.
  \begin{enumerate}
\item $$\til{Q}_1:\,\, {\xymatrix  @R=0.3cm  @C=0.6cm {a \ar@(ul,dl)[]_{\al}\ar[r]^{\beta}&x}}$$

\noindent with $\alpha^n\in I_A$ for $n\geq 2$, and $\alpha\beta \in I_A$. If there is an arrow $\delta:x \rightarrow y$ in $Q_A$, then $\beta\delta\in I_A$. Moreover, there are no arrows in $Q_A$ going out or coming in from the vertices of $\til{Q}_1$; or

\item $$\til{Q}_2:\,\,{\xymatrix  @R=0.3cm  @C=0.6cm {x\ar[r]^{\beta}&a \ar@(ur,dr)[]^{\al}}}$$

\noindent with $\alpha^n\in I_A$ for $n\geq 2$, and $\beta\alpha \in I_A$. If there is an arrow $\delta:y\fle x$  in $Q_A$, then $\delta \beta\in I_A$. Moreover, there are no arrows in $Q_A$ going out or coming in from the vertices of $\til{Q}_2$; or

\item  $$\til{Q}_3:\,\,\xymatrix  @R=0.4cm  @C=0.6cm {
&\bullet \ar@{.}[r] & \bullet  \ar[rd]^{\gamma_m}& & 	\\
1 \ar[ru]^{\gamma_1} \ar[rrr]^{\alpha} & & &\bullet \ar[r]^{\beta} &a
}$$

\noindent with $m\geq 1$,  $\alpha \beta \in I_A$ and $\gamma_1\dots \gamma_m\beta\notin I_A$ . If there is an arrow $\delta:a \rightarrow y$ in $Q_A$, then  $\gamma_1 \dots \gamma_m \beta\delta\in I_A$. If  there is an arrow $\lambda:\bullet \rightarrow 1$ in $Q_A$, then  $\lambda\gamma_1\dots \gamma_m\in I_A$. Moreover, the vertex $a$ is not the end point of any other arrow; or

\item  $$\til{Q}_4:\,\,\xymatrix  @R=0.4cm  @C=0.6cm {
& &\bullet \ar@{.}[r] & \bullet  \ar[rd]^{\gamma_m}&  	\\
a \ar[r]^{\beta} & \bullet \ar[ru]^{\gamma_1} \ar[rrr]^{\alpha} & & &1 &
}$$

\noindent with $m\geq 1$,  $\beta \alpha \in I_A$ and $\beta\gamma_1\dots \gamma_m\notin I_A$. If there is an arrow $\delta: y \rightarrow a$ in $Q_A$, then $\delta\beta \gamma_1\dots \gamma_m\in I_A$. If  there is an arrow $\lambda:1\rightarrow \bullet$ in $Q_A$, then  $\gamma_1\dots\gamma_m\lambda\in I_A$. Moreover, the vertex $a$ is not the start point of any other arrow.
\end{enumerate}
\end{enumerate}
\end{prop}
\begin{proof}
Let $M$ be an indecomposable $A$-module, not $\tau$-stable, and such that there is an irreducible morphism from $M$ to $\tau\,M$. Since  $M$ is not  $\tau$-stable then there is an integer $m$ such that $\tau^{m} M$ is either projective or injective.

Without loss of generality, we may assume that  $\tau^{m} M =I_a$ where $I_a$ is the injective corresponding to the vertex $a$ in $Q_A$. Moreover, with our notations, we have that there is an irreducible morphism from $I_a$ to $\tau I_a$.

Since $A$ is a string algebra, by \cite{BR} we know that $I_a=M(\overline{D_1}\;\overline{D_2})$ with $\overline{D_1}$ a direct string starting on a peak and $\overline{D_2}$ an inverse string ending on a peak.
Assume that $\tau I_a = M(D_1)$, where $D_1=\alpha_r\dots \alpha_2$ if $\overline{D_1}=\alpha_r\dots \alpha_1$.

Now,  depending on the string $D_1$, we shall analyze the possible almost split sequences starting in $M(D_1)$.

Firstly, assume that $D_1$ does not start on a peak and neither ends on a peak. Then the almost split sequences starting in $M(D_1)$ is as follows:

\[0\rightarrow M(\alpha_r\dots\alpha_2)\rightarrow M(C^{-1}\beta \alpha_r\dots \alpha_2)\oplus M(\alpha_r\dots\alpha_2\gamma^{-1}C')\rightarrow M(C^{-1}\beta\alpha_r\dots\alpha_2\gamma^{-1}C')\rightarrow 0.\]

\noindent Therefore $I_a=M(C^{-1}\beta\alpha_r\dots\alpha_2\gamma^{-1}C')$. Since  $I_a$ is injective then $l(C)=l(C')=0$. Moreover, since  $l(C)=0$, then $s(\beta)$ is not the start point of any other arrow in $Q_A$. Similarly, since  $l(C')=0$ then $s(\gamma)$ is not the start point of any other arrow in $Q_A$. Then we conclude that

\[I_a=M(\alpha_r\dots \alpha_1\overline{D_2})=M(\beta\alpha_r\dots \alpha_2\gamma^{-1})\]

\noindent with $\alpha_i=\beta$ for all $i=1, \dots, r$, $\beta^{r+1}\in I_A$, and  $\gamma\beta\in I_A$. Moreover, if there is an arrow  $\delta:\bullet\rightarrow s(\gamma)$ then $\delta\gamma\in I_A$. Hence in $Q_A$ there is a full subquiver of the following form:

\[\begin{array}{ccccc}
& &{\xymatrix  @R=0.3cm  @C=0.6cm {x\ar[r]^{\gamma}&a \ar@(ur,dr)[]^{\beta}}}&&
\end{array}
\]

\noindent with $\beta^{r+1}\in I_A$ for $,r\geq 1$ and $\gamma\be\in I_A$.

Secondly, suppose that $D_1$ starts and ends on a peak. In such a case $M(D_1)$ is injective, getting a contradiction to the fact that $\tau^{-1}M(D_1)=I_a$.

Now, suppose that $D_1$ does not start on a peak, but ends on a peak. Then the almost split sequence starting in $M(D_1)$ is as follows:

\[0\rightarrow M(\alpha_r\dots\alpha_2)\rightarrow M(C^{-1}\beta \alpha_r\dots \alpha_2)\oplus M(\alpha_r\dots\alpha_3)\rightarrow M(C^{-1}\beta\alpha_r\dots\alpha_3)\rightarrow 0.\]

\noindent Then $I_a=M(C^{-1}\beta\alpha_r\dots \alpha_3)$ with $l(C)=0$. Then there is a path of $r-1$ arrows, while $\overline{D_1}$ has $r$, a contradiction.

Finally, assume that $D_1$ starts on a peak and does not end on a peak. In this case, since $D_1$ is a direct string then the almost split sequence starting in $M(D_1)$ has only one  indecomposable middle term and it is as follows:

\[0\rightarrow M(\alpha_r\dots\alpha_2)\rightarrow  M(\alpha_r\dots\alpha_2\beta^{-1}C)\rightarrow M(C)\rightarrow 0\]

\noindent with $C$ a direct string  ending in a deep. Then $I_a=M(C)=M(\alpha_r\dots \alpha_1)$ is uniserial. Hence $Q_A$ has a subquiver of the form

 $$\xymatrix  @R=0.4cm  @C=0.6cm {
&\bullet \ar@{.}[r] & \bullet  \ar[rd]^{\alpha_2}& & 	\\
1 \ar[ru]^{\alpha_r} \ar[rrr]^{\beta} & & &\bullet \ar[r]^{\alpha_1} &a
}$$

\noindent with  $\beta\alpha_1\in I_A$. In case that there is an arrow $\lambda:\bullet \rightarrow 1$ then $\lambda\alpha_r\dots \alpha_2\in I_A$, because  $\alpha_r \dots \alpha_2$ starts on a peak.
Note that in this case,  $\alpha_i$ can be all trivial for $i=2, \dots, r$. In such a case, we have a subquiver as follows:

\[  {\xymatrix  @R=0.3cm  @C=0.6cm {\bullet \ar@(ul,dl)[]_{\beta}\ar[r]^{\alpha}&a}}\]

\noindent where $\be\al \in I_A$ (otherwise, $I_a\neq M(\alpha)$) and  $\beta^n\in I_A$ (in order to be a finite dimensional  algebra).

With a similar analysis as before and assuming that $\tau^{m} M$ is projective, we obtain the subquivers  (a) and (d).

For the converse, it is enough to show that for each configuration there is an indecomposable $A$-module $M$ and an irreducible morphism  from  $M$ to $\tau\, M$.
\end{proof}

As an immediate consequence of Proposition \ref{not}, we get the following corollary.

\begin{coro}\label{not2} With the notation introduced in Proposition \ref{not}, the following conditions hold.
\begin{enumerate}
  \item In $\til{Q}_1$ there are irreducible morphisms from $I_x$ to $\tau I_x$ and  from $\tau^{-1}P_a$ to $P_a$.
  \item In $\til{Q}_2$ there are irreducible morphisms from  $I_a$ to $\tau I_a$ and from $\tau^{-1}P_x$ to $P_x$.
  \item In  $\til{Q}_3$ there is an  irreducible morphism from $I_a$ to $\tau I_a$.
  \item In $\til{Q}_4$ there is an  irreducible morphism from   $\tau^{-1}P_a$ to $P_a$.
\end{enumerate}
\end{coro}

\begin{ej}
\emph{Let $A$ be the algebra given by the presentation}

 $$\xymatrix  @R=0.4cm  @C=0.6cm {
&2  \ar[rd]^{\gamma_2}& & 	\\
1 \ar[ru]^{\gamma_1} \ar[rr]^{\alpha}  & &3 \ar[r]^{\beta} &4
}$$

\noindent \emph{with  $I=<\al\be>$. Let $\Gamma$ be a component of $\Gamma_{A}$  having the injective $I_4$. Then $\Gamma$ is as follows:}

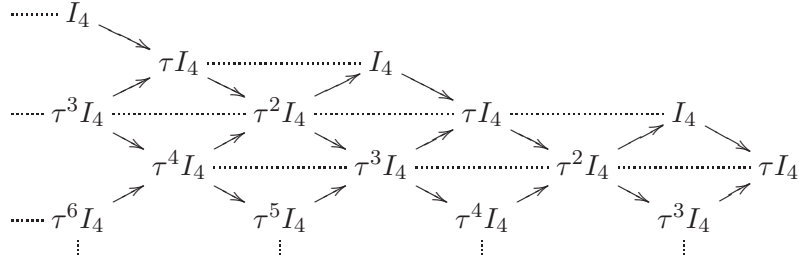
\begin{displaymath}
\xymatrix  @R=0.1cm  @C=0.4cm {
\ar@{.}[r]& I_4\ar[rd]&&&&&&&\\
&& \tau I_4\ar[rd]\ar@{.}[rr] &&I_4\ar[rd] &&&&\\
\ar@{.}[r]& \tau^3I_4\ar[rd]\ar[ru]\ar@{.}[rr]&&\tau^2I_4\ar[rd]\ar[ru]\ar@{.}[rr]&& \tau I_4\ar[rd]\ar@{.}[rr]&& I_4\ar[rd]&\\
&& \tau^4I_4\ar[rd]\ar[ru]\ar@{.}[rr]&& \tau^3I_4\ar[rd]\ar[ru]\ar@{.}[rr]&&\tau^2I_4\ar[rd]\ar[ru]\ar@{.}[rr]&&\tau I_4\\
\ar@{.}[r] &\tau^6I_4\ar[ru]\ar@{.}[d]&&\tau^5I_4\ar[ru]\ar@{.}[d]&& \tau^4I_4\ar[ru]\ar@{.}[d]&& \tau^3I_4\ar@{.}[d]\ar[ru]&\\
&&&&&&&&}
\end{displaymath}

\noindent \emph{where we identify the modules in $\Gamma$ which are the same.}
\emph{Observe all the modules $M$ that belong to the sectional path starting in $I_4$  have the property that there is an irreducible morphism from $M$ to $\tau M$.}
\end{ej}

Now, we concentrate our attention in the algebras which have an irreducible morphism from $M$ to $\tau M$, where $M$ is an indecomposable $\tau$-stable module.

\begin{prop}
Let $A=kQ_A/I_A$ be a string algebra. The following conditions are equivalent.
\begin{enumerate}
  \item There is a $\tau$-stable indecomposable $A$-module $M$ with $\alpha'(M)=1$ and an irreducible morphism $M\rightarrow \tau M$.
  \item The quiver $Q_A$ contains one of the following full subquivers:
  \begin{enumerate}
\item[i)] \begin{displaymath}
\xymatrix  @R=0.3cm  @C=0.6cm {
	1 \ar@(ul,dl)[]_{\al}\ar[r]^{\beta}& 2}
\end{displaymath}

\noindent with $\al^2\in I_A$ and $\alpha\beta\notin I_A$. Moreover, there are no arrows coming in the vertex $1$, if there is an arrow $\lambda:2\rightarrow \bullet$ then $\beta\lambda\in I_A$ and $2$ is not the end point of any other arrow; or
\item[ii)] {\begin{displaymath}
\xymatrix  @R=0.3cm  @C=0.6cm {
	1 \ar@(ul,dl)[]_{\al}&2\ar[l]^{\beta}
	}
\end{displaymath}

\noindent  with  $\al^2\in I_A$ and $\beta\alpha\notin I_A$. Moreover, there are no arrows going out from the vertex  $1$, if there is an arrow $\lambda:\bullet\rightarrow 2$ then  $\lambda\beta\in I_A$ and $2$ is not the starting  point of any other arrow.}
\end{enumerate}
\end{enumerate}
\end{prop}

\begin{proof}
Let $M$ be an indecomposable  $A$-module as in (1). Since $\alpha'(M)=1$ by  \cite{BR},
we get  that $M=M(\gamma_1^{-1}\dots \gamma_r^{-1})=M(B_2^{-1})$ and  $\tau M= N(\beta_0)=M(\delta_s^{-1}\dots \delta_1^{-1}\beta_0\gamma_1^{-1}\dots \gamma_r^{-1})=M(B_1^{-1}\beta_0 B_2^{-1})$. Observe that $\tau M$ can not be the starting of an almost split sequence with indecomposable middle term. Hence, the almost split sequence starting in $\tau M$ has two indecomposable middle terms.

Now, we shall built such a sequence. We know that the string $B_1^{-1}\beta B_2^{-1}$ ends on a peak. Then we analyze two cases:
\begin{enumerate}
\item[(a)] if  $B_1^{-1}\beta_0 B_2^{-1}$ starts on a peak; or
\item[(b)] if  $B_1^{-1}\beta_0 B_2^{-1}$ does not start on a peak.
\end{enumerate}

Assume that (a) holds, then there is no $\lambda\in Q_1$ such that $\lambda B_1^{-1}\beta_0 B_2^{-1}$ is a string. Since $\tau^{-1}M$ is not injective, then $s\geq1$. The almost split sequence starting in $\tau M$ is as follows:

\[0\rightarrow M(B)\rightarrow M(B_1^{-1})\oplus M(\delta_{s-1}^{-1}\dots \delta_1^{-1}\beta_0 B_2^{-1})\rightarrow M(\delta_{s-1}^{-1}\dots \delta_{1}^{-1})\rightarrow 0.\]

\noindent Therefore

\begin{equation}\label{eqn1}
M=M(B_2^{-1})=M(\delta_{s-1}^{-1}\dots \delta_{1}^{-1}).
\end{equation}

\noindent  Hence $\gamma_1^{-1}\dots \gamma_r^{-1}=\delta_{s-1}^{-1}\dots \delta_{1}^{-1}$. Then $r=s-1$ and $\gamma_{i}=\delta_{s-i}$. If $r\geq 1$
then there is a subquiver of the form:

$$
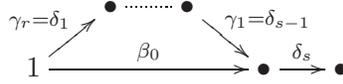
\xymatrix  @R=0.4cm  @C=0.6cm {
&\bullet \ar@{.}[r] & \bullet  \ar[rd]^{\gamma_1=\delta_{s-1}}& & 	\\
1 \ar[ru]^{\gamma_r =\delta_{1}} \ar[rrr]^{\beta_0} & & &\bullet \ar[r]^{\delta_s} & \bullet
}$$

\noindent where $\beta_0\delta_s\in I_A$ since  $\gamma_1\delta_s \notin I_A$.
Observe that since the string $B_2^{-1}$ ends on a peak and $B_2^{-1}=\delta_{s-1}^{-1}\dots\delta_{1}^{-1}$ then $B_1^{-1}=\delta_s^{-1}\dots \delta_1^{-1}$ also ends on a peak.

On the other hand, since we assume (a) then $B_1^{-1}\beta_0 B_2^{-1}$ starts on a peak and therefore $B_1^{-1}$ starts on a peak. Then $M(B_1^{-1})$ is injective since $B_1^{-1}$ starts and ends on a peak, a contradiction to the fact that $M$ is $\tau$-stable and $M(B_1^{-1})=\tau\,M$.

Then $r=0$ and $B_2^{-1}=e(\beta_0)$ and $B_1^{-1}=\delta_1$. From (\ref{eqn1}) we get that $B_2^{-1}=e(\beta_0)=s(\beta_0)$, and therefore $\beta_0$ is a loop. Since $e(\beta_0)$ is not the end point of any other arrow, we have that if $\beta_0\delta_1\in I_A$ then $M(\delta_1)$ is injective getting a contradiction. Thus $\beta_0\delta_1\notin I_A$ and then $\beta_0^{2}\in I_A$. Then in $Q_A$ we have a subquiver of the form:

\[
\xymatrix  @R=0.3cm  @C=0.6cm {
	1 \ar@(ul,dl)[]_{\beta_0}\ar[r]^{\delta_1}&2}
\]

\noindent such that $\beta_0^{2}\in I_A$, $\beta_0\delta_1\notin I_A$, there are no arrows coming in or going out of the vertex $1$, if there is an arrow  $\lambda:e(\delta_1)\rightarrow \bullet$ then $\delta_1\lambda\in I_A$ and $e(\delta_1)$ is not the end point of any other arrow, since the string $B_1^{-1}\beta_0 B_2^{-1}=\delta_1^{-1}\beta_0$ starts on a peak.

Assume now, that $B_1^{-1}\beta_0 B_2^{-1}$ satisfies (b). That is, there is a an arrow $\lambda\in Q_1$ such that  $\lambda B_1^{-1}\beta_0 B_2^{-1}$. In this case, the almost split sequence starting in $B_1^{-1}\beta_0 B_2^{-1}$ is as follows:

\[0\rightarrow M(B_1^{-1}\beta_0 B_2^{-1})\rightarrow M(B_1^{-1})\oplus M(D^{-1}\lambda B_1^{-1}\beta_0 B_2^{-1})\rightarrow M(D^{-1}\lambda B_1^{-1})\rightarrow 0.\]

Then $M=M(B_2^{-1})=M(D^{-1}\lambda B_1^{-1})$. In this case, we obtain that $B_2^{-1}=B_1\lambda^{-1}D$. Thus, we deduce that $B_1$ and $D$ are trivial and  $B_2^{-1}=\lambda^{-1}$. Since $D$ is trivial, $s(\lambda)$ is not the starting point of any other arrow. Similarly, since $B_1$ is trivial then $s(\beta_0)$ is not the starting point of any other arrow. Furthermore, since the string $\lambda\beta_0 \lambda^{-1}$ is defined, $\beta_0$ is a loop and  $\beta_0^{2}\in I_A$ because $\beta_0\lambda\notin I_A$. Then we have a subquiver as follows:

\[\xymatrix  @R=0.3cm  @C=0.6cm {
	1 \ar@(ul,dl)[]_{\beta_0}&2\ar[l]^{\lambda}
	}\]

\noindent with $\beta_0^2\in I_A$, $\lambda\beta_0\notin I_A$, and where there are not arrows going out the vertex $1$, and if there is $\rho:\bullet\rightarrow 2$ then s $\rho\lambda\in I_A$ and  $2$ is not the starting point of any other arrow.
\end{proof}

\section{On the composition of three irreducible morphisms}

We shall prove several lemmas in order to prove the main result of this work.

\begin{lem}\label{W-A3-notinjective} Let $A$ be a string algebra. A configuration of almost split sequences as follows:

\begin{equation} \label{1}
\xymatrix  @R=0.3cm  @C=0.5cm {
X\ar[dr]_{f_1}\ar@{.}[rr] &&Z\ar[dr]&&&&\\
&Y\ar[dr]_{f_2}\ar@{.}[rr]\ar[ru]_{f}&&U\ar[dr]&&&\\
&&W\ar[dr]_{g_1}\ar@{.}[rr]\ar[ru]_{f_3}&&\tau^{-1} W\ar[dr]&&\\
&&&L\ar[dr]_{g_2}\ar@{.}[rr]\ar[ru]_{s_1}&&\tau^{-1}L\ar[dr]&\\
&&&&\tau L\ar[dr]_{g_3}\ar@{.}[rr]\ar[ru]_{s_2}&&L\\
&&&&&W\ar[ru]_{g_1}&\\}
\end{equation}
\vspace{.05in}

\noindent is a forbidden configuration in $\Gamma_A$.
\end{lem}

\begin{proof} Since ${f}$ is an epimorphism then $g_1: W\rightarrow L$ so is. The $A$-modules
$X$ and $Z$ are the end points of an almost split sequence with indecomposable middle term $Y$. Then $Y=N(\beta_0)=M(\delta_s^{-1}\dots \delta_1^{-1}\beta_0\gamma_1^{-1}\dots\gamma_r^{-1})=M(C)$ with $C$ a string that  starts in a deep and ends in a peak. Moreover,  $X=M(\gamma_1^{-1}\dots\gamma_r^{-1})$ and $Z=M(\delta_s^{-1}\dots \delta_1^{-1})$.

By \cite{BR} and from (\ref{1}), we know that $f_2, g_1, g_2$ and $g_3$ are the irreducible morphisms obtained by analyzing the beginning of the string corresponding to the domain of such morphisms.

We start considering the case that $C$ starts on a peak.
Then $C$ starts and ends on a peak. Since $Y$ is not injective, then $s\geq 1$. Hence, $W=M(\delta_{s-1}^{-1}\dots \delta_1^{-1}\beta_0\gamma_1^{-1}\dots\gamma_r^{-1})$, $U=M(\delta_{s-1}^{-1}\dots \delta_1^{-1})$. Consider $D_1= \delta_{s-1}^{-1}\dots \delta_1^{-1}\beta_0\gamma_1^{-1}\dots\gamma_r^{-1}$. Then $W=M(D_1)=M(D_1^{-1})$.

Since  $g_1$ is an epimorphism, then the string corresponding to $W$ starts on a peak (otherwise, $g_1$ is a monomorphism). Therefore,  $\delta_{s-1}^{-1}\dots \delta_1^{-1}\beta_0\gamma_1^{-1}\dots\gamma_r^{-1}$ is a string that starts and ends on a peak. Now, since $W$ is not injective then  $s\geq 2$.
Thus, $L=M(\delta_{s-2}^{-1}\dots \delta_1^{-1}\beta_0\gamma_1^{-1}\dots\gamma_r^{-1})$ and  $\tau^{-1} W= M(\delta_{s-1}^{-1}\dots \delta_1^{-1})$.

Now, we analyze how is the string corresponding to $\tau L$. In order to do that, we may consider how is the beginning of the string corresponding to $L$.
Assume that $\delta_{s-2}^{-1}\dots \delta_1^{-1}\beta_0\gamma_1^{-1}\dots\gamma_r^{-1}$ does not start in a peak. Then $\tau L= M(\nu_t^{-1}\dots \nu_1^{-1}\beta_1\delta_{s-2}^{-1}\dots \delta_1^{-1}\beta_0\gamma_1^{-1}\dots\gamma_r^{-1})$ with $t\geq 1$ or
$\tau L=M(\beta_1\delta_{s-2}^{-1}\dots \delta_1^{-1}\beta_0\gamma_1^{-1}\dots\gamma_r^{-1}).$

Assume that $\tau L= M(\nu_t^{-1}\dots \nu_1^{-1}\beta_1\delta_{s-2}^{-1}\dots \delta_1^{-1}\beta_0\gamma_1^{-1}\dots\gamma_r^{-1})$ with $t\geq 1$. Suppose that $\nu_t^{-1}\dots \nu_1^{-1}\beta_1\delta_{s-2}^{-1}\dots \delta_1^{-1}\beta_0\gamma_1^{-1}\dots\gamma_r^{-1}$ starts on a peak.
Since $t\geq1$, then there exists an irreducible morphism  from $\tau L$ to $M(\nu_{t-1}^{-1}\dots \nu_1^{-1}\beta_1\delta_{s-2}^{-1}\dots \delta_1^{-1}\beta_0\gamma_1^{-1}\dots\gamma_r^{-1})$ and,  by construction, this module is $W$.

Since $D_1$ has only one direct arrow then  $D_1\neq \nu_{t-1}^{-1}\dots \nu_1^{-1}\beta_1\delta_{s-2}^{-1}\dots \delta_1^{-1}\beta_0\gamma_1^{-1}\dots\gamma_r^{-1}$.
If $D_1^{-1}= \nu_{t-1}^{-1}\dots \nu_1^{-1}\beta_1\delta_{s-2}^{-1}\dots \delta_1^{-1}\beta_0\gamma_1^{-1}\dots\gamma_r^{-1}$ and since $s\geq2$ then the arrows $\gamma_i$ are trivial. Thus $t\geq2$. Then $D_1^{-1}$ has length $s$, while $\nu_{t-1}^{-1}\dots \nu_1^{-1}\beta_1\delta_{s-2}^{-1}\dots \delta_1^{-1}\beta_0\gamma_1^{-1}\dots\gamma_r^{-1}$ has at least length $s+1$, a contradiction.
Then $\nu_t^{-1}\dots \nu_1^{-1}\beta_1\delta_{s-2}^{-1}\dots \delta_1^{-1}\beta_0\gamma_1^{-1}\dots\gamma_r^{-1}$ does not start on a peak.
 Therefore there is an irreducible morphism as follows:
$$\tau L \rightarrow M(\lambda_k^{-1}\dots \lambda_1^{-1}\beta_2\nu_t^{-1}\dots \nu_1^{-1}\beta_1\delta_{s-2}^{-1}\dots \delta_1^{-1}\beta_0\gamma_1^{-1}\dots\gamma_r^{-1})$$
\noindent  where the length of such a string is different from the length of $D_1$ (and $D_1^{-1}$), proving that this case is not possible.

Now, consider that $\tau L=M(\beta_1\delta_{s-2}^{-1}\dots \delta_1^{-1}\beta_0\gamma_1^{-1}\dots\gamma_r^{-1})$. If $\beta_1\delta_{s-2}^{-1}\dots \delta_1^{-1}\beta_0\gamma_1^{-1}\dots\gamma_r^{-1}$ does not start on a peak, then we can add an arrow and the string should have at least length $r+s+1$, while the string corresponding to $W$ has length $r+s$, a contradiction. If $\beta_1\delta_{s-2}^{-1}\dots \delta_1^{-1}\beta_0\gamma_1^{1}\dots\gamma_r^{-1}$ starts on a peak, since $\tau L$ is not injective then $s\geq3$. Therefore, there is an irreducible morphism from $\tau L$ to $M(\delta_{s-3}^{-1}\dots \delta_1^{-1}\beta_0\gamma_1^{-1}\dots\gamma_r^{-1})$  where the length of such a string and the string corresponding to $W$ are different,  proving that this case is not possible.

Assume that the string corresponding to $L$ starts on a peak. Then $g_2$ is an epimorphism and the string $\delta_{s-2}^{-1}\dots \delta_1^{-1}\beta_0\gamma_1^{-1}\dots\gamma_r^{-1}$ starts and ends in a peak.  Hence, $s\geq 3$, otherwise $L$ is injective. Then $\tau L=M(\delta_{s-3}^{-1}\dots \delta_1^{-1}\beta_0\gamma_1^{-1}\dots\gamma_r^{-1})$.

By Lemma \ref{liu-cha}, since $g_1$ and $g_2$ are epimorphisms then $g_3$ is a monomorphism.
Since the existence of $g_3$ is due to the fact of how is the beginning of the string corresponding to $\tau L$, then $\delta_{s-3}^{-1}\dots \delta_1^{-1}\beta_0\gamma_1^{-1}\dots\gamma_r^{-1}$ does not start in a peak. Thus, there is an irreducible morphism from
$\tau L$ to  $M(\lambda_k^{-1}\dots \lambda_1^{-1}\beta_1\delta_{s-3}^{-1}\dots \delta_1^{-1}\beta_0\gamma_1^{-1}\dots\gamma_r^{-1})$.
It is clear that this string is different from $D_1$, since $D_1$ has only one arrow. Then we have that $D_1^{-1}=\lambda_k^{-1}\dots \lambda_1^{-1}\beta_1\delta_{s-3}^{-1}\dots \delta_1^{-1}\beta_0\gamma_1^{-1}\dots\gamma_r^{-1}$. Since $s\geq3$, this implies that the arrows $\gamma_i$ are trivial. If $s\geq4$ then $D_1^{-1}$ has at least three arrows, a contradiction, because the obtained string has two arrows. Therefore $s=3$ and, in consequence, $k=1$. Then  $\lambda_1^{-1}\beta_1\beta_0=\beta_0^{-1}\delta_1\delta_2$, and we get that $\beta_0=\delta_2$ and $\beta_0^{-1}=\lambda_1^{-1}$ is a string that starts in a deep. Note that the string $\delta_3^{-1}\delta_2^{-1}=\delta_3^{-1}\beta_0^{-1}$ is defined, hence we get a contradiction. In conclusion,  $C$ can not start on a peak.

In case that  $C$ does not start on a peak, with similar arguments as above, we can conclude that there is not possible to have a configuration of almost split sequences as in $(\ref{1})$, whenever $A$ is a string algebra.
\end{proof}

\begin{lem} \label{W-injectivo 2}
A configuration of almost split sequences as follows:

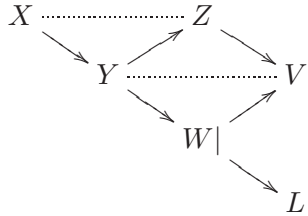
\[
\xymatrix  @R=0.3cm  @C=0.6cm {
X\ar[dr]\ar@{.}[rr]&&Z\ar[dr]&\\
&Y\ar[dr]\ar@{.}[rr]\ar[ru]&&V\\
&&W|\ar[dr]\ar[ru]&\\
&&&L}
\]

\noindent with  $W$ an indecomposable injective $A$-module, $L\not\simeq \tau W$ an indecomposable  injective $A$-module such that there is an irreducible morphism from $L$ to $\tau L$ is not a possible configuration in $\Gamma_A$.
\end{lem}
\begin{proof}
Since $L$ is not injective then $\tau^{-1}L$ is defined. The existence of an irreducible morphism from $L$ to $\tau L$, implies the existence of an  irreducible morphism from $\tau L$ to $\tau^{-1}L$. Moreover, there is an  irreducible morphism from $\tau L$ to $W$. Since $L\not\simeq \tau W$ then  $\alpha'(\tau L)=2$.

Assume that $\alpha'(L)=2$. Then there is a configuration of almost split sequences as follows:

\begin{equation} \label{injectiveW}
\xymatrix  @R=0.3cm  @C=0.6cm {
 X\ar[dr] && Z \ar[dr]&&  &\\
	& Y\ar[dr]\ar[ur]\ar@{.}[rr]&&V&& \\
&& W|\ar[dr]^{g}\ar[ur]&&|P\ar[dr]&\\
& \tau L\ar[dr]\ar[ur]\ar@{.}[rr]&&L\ar[dr]\ar[ur]^{f}\ar@{.}[rr]&& \tau^{-1}L\\
L\ar[ur]\ar[dr]_{f}\ar@{.}[rr]&& \tau^{-1}L\ar[ur]&&\tau L\ar[ur]&\\
& |P \ar[ur]&&  && }
\end{equation}

\noindent Since $W$ is injective, then $g$ is an epimorphism. By $(\ref{injectiveW})$, we know that $f$ is an epimorphism. On the other hand, since $P$ is projective then $f$ is a monomorphism, a contradiction. Hence, $\alpha'(L)=1$.

Now, we analyze the string corresponding to such modules. Since $X$ and $Z$ are the start and end terms of an almost split sequence with indecomposable middle term $Y$,  then $Y=N(\beta_0)=M(\delta_s^{-1}\dots \delta_1^{-1}\beta_0\gamma_1^{-1}\dots\gamma_r^{-1})=M(C)$ with $C$ a string that starts on a deep and ends on a peak. Moreover, $X=M(\gamma_1^{-1}\dots\gamma_r^{-1})$ and $Z=M(\delta_s^{-1}\dots \delta_1^{-1})$.

Firstly, assume that $C$  starts on a peak. Then $C$  starts and ends on a peak, and therefore, since $Y$ is not injective then $s\geq 1$. Hence, $W=M(\delta_{s-1}^{-1}\dots \delta_1^{-1}\beta_0\gamma_1^{-1}\dots\gamma_r^{-1})$ and $U=M(\delta_{s-1}^{-1}\dots \delta_1^{-1})$. Since $W$ is injective then $s-1=0$, $W=M(\beta_0\gamma_1^{-1}\dots\gamma_r^{-1})=M(D_1)$ and $D_1$ is a string that  starts and ends on a peak. If $r=0$, then $W/\mbox{soc} W$ is indecomposable, a contradiction to  $(\ref{injectiveW})$.  Thus, $r\geq 1$ and $L=M(\gamma_2^{-1}\dots\gamma_r^{-1})$ (if $r=1$, $L$ is simple). Since  $\alpha'(L)=1$, then $\tau L= M(\lambda_k^{-1}\dots \lambda_1^{-1}\beta_1\gamma_2^{-1}\dots\gamma_r^{-1})$ and the string corresponding to $\tau L$ starts in a deep and ends on a peak. Now, we analyze  the beginning of the string corresponding to $\tau L$ to determine the codomain of the irreducible morphisms whose domain is $\tau L$.

If $\lambda_k^{-1}\dots \lambda_1^{-1}\beta_1\gamma_2^{-1}\dots\gamma_r^{-1}$ starts on a peak, since  $\tau L$ is not injective then $k\geq 1$. Then there is an irreducible morphism from $\tau L$ to $M(\lambda_{k-1}^{-1}\dots \lambda_1^{-1}\beta_1\gamma_2^{-1}\dots\gamma_r^{-1})$ and this module is $W$, therefore injective. Then we get that $W=M(\beta_1\gamma_2^{-1}\dots\gamma_r^{-1})$ and $\beta_1\gamma_2^{-1}\dots\gamma_r^{-1}$ has length $r$ a contradiction.

If $\lambda_k^{-1}\dots \lambda_1^{-1}\beta_1\gamma_2^{-1}\dots\gamma_r^{-1}$ does not start on a peak, then  there is an irreducible morphism from $\tau L$ to $M(\epsilon_l^{-1}\dots \epsilon_1^{-1}\beta_2\lambda_k^{-1}\dots \lambda_1^{-1}\beta_1\gamma_2^{-1}\dots\gamma_r^{-1})$, and this module should be injective. Hence, there is an irreducible morphism from  $\tau L$ to $M(\beta_2\beta_1\gamma_2^{-1}\dots\gamma_r^{-1})$. Clearly, $\beta_2\beta_1\gamma_2^{-1}\dots\gamma_r^{-1}$ is not equal to $D_1$. Similarly, we can see that $\beta_2\beta_1\gamma_2^{-1}\dots\gamma_r^{-1}$ is not equal to $D_1^{-1}$.

Secondly, assume that $C$ does not  start on a peak. In this case, we have that $W=M(\nu_t^{-1}\dots \nu_{1}^{-1}\beta_1\delta_s^{-1}\dots \delta_1^{-1}\beta_0\gamma_1^{-1}\dots\gamma_r^{-1})$. Since $W$ is injective, then $t=0$, $s=0$ and  $W=M(\beta_1\beta_0\gamma_1^{-1}\dots\gamma_r^{-1})=M(D_2)$ with $D_2$ is a string that starts and ends on a peak. Then $r\geq1$, otherwise, $W/\mbox{soc} W$ is indecomposable.  Then $L=M(\gamma_2^{-1}\dots\gamma_r^{-1})$. Since  $\alpha'(L)=1$, then $\tau L= M(\lambda_k^{-1}\dots \lambda_1^{-1}\beta_1\gamma_2^{-1}\dots\gamma_r^{-1})$ and the string corresponding to $\tau L$ starts in a deep and ends on a peak.

Again, if we analyze the beginning of the string corresponding to $\tau L$ to determine the codomain of the irreducible morphisms with domain $\tau L$,  we can discard the case  with similar arguments as before.
\end{proof}

\begin{lem}\label{lema W no iny 2}
A configuration of almost split sequences as follows:

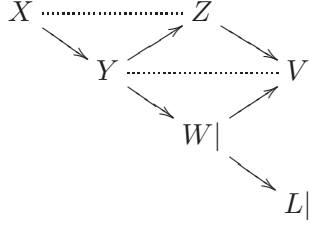
\[
\xymatrix  @R=0.3cm  @C=0.6cm {
X\ar[dr]\ar@{.}[rr]&&Z\ar[dr]&\\
&Y\ar[dr]\ar@{.}[rr]\ar[ru]&&V\\
&&W|\ar[dr]\ar[ru]&\\
&&&L|}
\]

\noindent with  $W$ and  $L$ indecomposable injective $A$-modules such that there is an irreducible morphism from $L$ to $\tau L$ is a forbidden configuration in $\Gamma_A$.
\end{lem}
\begin{proof}
Since there are not morphisms from an injective to a  projective, then $\alpha'(L)=1$. Moreover, $\tau L$ is not projective, then $\tau^2 L$ is defined and there is an irreducible morphism from $\tau L$ to $\tau^2 L$. Since $W$ is injective, then $\tau^2 L \not \simeq W$ and therefore, $\alpha'(\tau L)=2$. Then there is a configuration of almost split sequences as follows:

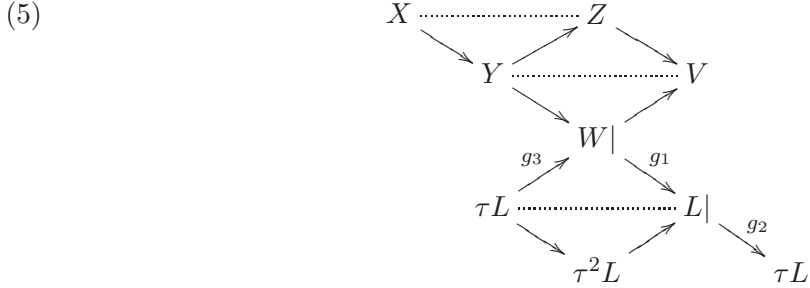
\begin{equation} \label{injective-W2}
\xymatrix  @R=0.3cm  @C=0.6cm {
X\ar[dr]\ar@{.}[rr]&&Z\ar[dr]&&\\
&Y\ar[dr]\ar@{.}[rr]\ar[ru]&&V&\\
&&W|\ar[dr]^{g_1}\ar[ru]&&\\
&\tau L\ar@{.}[rr]\ar[dr]\ar[ru]^{g_3}&&L|\ar[dr]^{g_2}&\\
&&\tau^2 L\ar[ru]&&\tau L}
\end{equation}

Since $W\rightarrow W$ is a cycle and $g_1, g_2$ are epimorphisms then by  Lemma \ref{liu-cha}, $g_3$ is a monomorphism.

Since $X$ and $Z$ are the start and end terms of an almost split sequence with indecomposable middle term $Y$,  then  $Y=N(\beta_0)=M(\delta_s^{-1}\dots \delta_1^{-1}\beta_0\gamma_1^{-1}\dots\gamma_r^{-1})=M(C)$ with  $C$ a string that starts in a deep and ends on a peak.  Moreover,  $X=M(\gamma_1^{-1}\dots\gamma_r^{-1})$ and $Z=M(\delta_s^{-1}\dots \delta_1^{-1})$.

Assume that $C$ starts on a peak. Then $C$ starts and ends on a peak and therefore, since $Y$ is not injective then $s\geq 1$. Hence, $W=M(\delta_{s-1}^{-1}\dots \delta_1^{-1}\beta_0\gamma_1^{-1}\dots\gamma_r^{-1})$, and  $U=M(\delta_{s-1}^{-1}\dots \delta_1^{-1})$. Since  $W$ is injective, then $s-1=0$, $W=M(\beta_0\gamma_1^{-1}\dots\gamma_r^{-1})=M(D_1)$ and  $D_1$ starts and ends on a peak. If $r=0$, then $W/\mbox{soc} W$ is indecomposable, a contradiction with $(\ref{injective-W2})$. Then $r\geq 1$ and $L=M(\gamma_2^{-1}\dots\gamma_r^{-1})$. Since $L$ is injective, but not simple
then $r\geq 2$ and $\tau L= L/\mbox{soc}\,L= M(\gamma_3^{-1}\dots\gamma_r^{-1})$.

Since the irreducible morphism from $\tau L$ to  $W$ is a monomorphism, then $\gamma_3^{-1}\dots\gamma_r^{-1}$ either does not start on a peak or does not end on a peak. In the former case, there is an irreducible morphism from $\tau L$ to $M(\lambda_k^{-1}\dots \lambda_1^{-1}\beta_1\gamma_3^{-1}\dots\gamma_r^{-1})$ and this module is $W$ and therefore, injective. Then it is of the form  $M(\beta_1\gamma_3^{-1}\dots\gamma_r^{-1})$ but the string corresponding to this module has length $r-1$, a contradiction. In the latter case, there is an irreducible morphism from $\tau L$ to  $M(\gamma_3^{-1}\dots\gamma_r^{-1}\beta_2^{-1}\epsilon_1\dots \epsilon_l)$ and this module must be injective. Then it is of the form  $M(\gamma_3^{-1}\dots\gamma_r^{-1}\beta_2^{-1})$ but the string corresponding to this  module is of length $r-1$,  a contradiction.

Assume that $C$ does not start on a peak. Then $W= M(\nu_t^{-1}\dots \nu_{1}^{-1}\beta_1\delta_s^{-1}\dots \delta_1^{-1}\beta_0\gamma_1^{-1}\dots\gamma_r^{-1})$. Since  $W$ is  injective, then $t=0$, $s=0$ and  $W=M(\beta_1\beta_0\gamma_1^{-1}\dots\gamma_r^{-1})=M(D_2)$ with  $D_2$ a string that starts and ends on a peak. Then $r\geq1$, otherwise, $W/\mbox{soc} W$ is indecomposable. Then $L=M(\gamma_2^{-1}\dots\gamma_r^{-1})$. Since  $L$ is  injective but not simple,
then  $r\geq 2$ and  $\tau L= L/\mbox{soc}\,L= M(\gamma_3^{-1}\dots\gamma_r^{-1})$. With similar arguments as before, we get that this case is not possible, proving the lemma.
\end{proof}

\begin{lem}\label{lema W no iny 3}
A configuration of almost split sequences as follows:

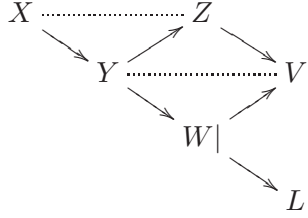
\[
\xymatrix  @R=0.3cm  @C=0.6cm {
X\ar[dr]\ar@{.}[rr]&&Z\ar[dr]&\\
&Y\ar[dr]\ar@{.}[rr]\ar[ru]&&V\\
&&W|\ar[dr]\ar[ru]&\\
&&&L}
\]

 \noindent with $W$ an indecomposable injective $A$-module, $L\simeq  \tau W$ and  $\alpha'(L)=2$ is not a possible configuration in $\Gamma_A$.
\end{lem}
\begin{proof}
Since $\alpha'(L)=2$ and $W$ is injective, then there is an irreducible morphism from  $L$ to a projective $A$-module  $P$, where $P \not \simeq Y$. Then there is a configuration of almost split sequences as follows:

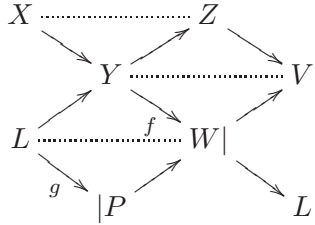
\[
\xymatrix  @R=0.3cm  @C=0.6cm {
X\ar[dr]\ar@{.}[rr]&&Z\ar[dr]&\\
&Y\ar[dr]_f\ar@{.}[rr]\ar[ru]&&V\\
L\ar[dr]_{g}\ar@{.}[rr]\ar[ru]&&W|\ar[dr]\ar[ru]&\\
&|P\ar[ru]&&L}
\]

\noindent Since $g$ is a monomorphism, then $f$ is a monomorphism.

Since $X$ and $Z$ are the end points of an almost split with indecomposable middle term $Y$, then $Y=N(\beta_0)=M(\delta_s^{-1}\dots \delta_1^{-1}\beta_0\gamma_1^{-1}\dots\gamma_r^{-1})=M(C)$ with $C$ a string that starts in a deep and ends in a peak.
 Since $f$ is a monomorphism, then the string corresponding to $Y$ does not start on a peak. Then  $W=M(\lambda_k^{-1}\dots\lambda_1^{-1}\beta_1\delta_s^{-1}\dots \delta_1^{-1}\beta_0\gamma_1^{-1}\dots\gamma_r^{-1})$. Since $W$ is injective, then $k=s=0$ and  $W=M(\beta_1\beta_0\gamma_1^{-1}\dots\gamma_r^{-1})$.

On the other hand, since $W/\mbox{soc}\,W$ is not indecomposable, then $r\geq 1$ and $L=M(\gamma_2^{-1}\dots \gamma_r^{-1})$. Moreover, there is an irreducible morphism from $L$ to $Y$.

If $r=1$, then $L=M(s(\gamma_1))$ is simple. There are irreducibles morphisms from $L$ to modules of the form $M(\epsilon_l^{-1}\dots \epsilon_1^{-1}\beta_2)$. The strings corresponding to such  modules are equal to $C$ or $C^{-1}$. In any case, it is a contradiction.

Now, assume that $r\geq2$. If the string corresponding to  $L$ starts or ends on a peak, then there is an irreducible morphism from  $L$ to a module whose string has length $r-3$ while $C$ has length $r+2$, a contradiction.

If  $\gamma_2^{-1}\dots \gamma_r^{-1}$ does not start on a peak, then there is an irreducible morphism from  $L$ to $M(\lambda_k^{-1}\dots \lambda_1^{-1}\beta_2\gamma_2^{-1}\dots \gamma_r^{-1})$ and this module should be $Y$. If $$\lambda_k^{-1}\dots \lambda_1^{-1}\beta_2\gamma_2^{-1}\dots \gamma_r^{-1}=C=\beta_1\gamma_1^{-1}\dots \gamma_r^{-1}$$ then $k=0$ and $r=1$. Hence $\beta_2=\beta_1\gamma_1^{-1}$, a contradiction. Now, if  $$\lambda_k^{-1}\dots \lambda_1^{-1}\beta_2\gamma_2^{-1}\dots \gamma_r^{-1}=C^{-1}=\gamma_r\dots \gamma_1\beta_1^{-1}$$ then $k=0$ and  $r=2$ getting a contradiction with the length of the strings.

Finally, if $\gamma_2^{-1}\dots \gamma_r^{-1}$ does not end on a peak, then there is an irreducible morphism from $L$ to $M(\gamma_2^{-1}\dots \gamma_r^{-1}\beta_2^{-1}\epsilon_1\dots\epsilon_l)$ and this module should be $Y$. With a similar analysis as before, we get that this case is not possible.
\end{proof}

\begin{lem} \label{Vicky}
Let $A$ be a string algebra, and  $\Gamma$ be a component of $\Gamma_A$. Let $I$ be
an injective (non-projective)  $A$-module such that there exists an irreducible
morphisms from $I$ to $\tau I$ with $I\in \Gamma$. Then, there are not three irreducible morphisms between indecomposable modules $f_1:X\rightarrow Y$, $f_2:Y\rightarrow W$ and $f_3:W\rightarrow V$ in $\Gamma$ such that $f_3f_3f_1\in \Re^6\backslash\Re^7$ and a configuration as follows:

\begin{equation}\label{eqnvicky}
\xymatrix  @R=0.3cm  @C=0.6cm {
X\ar[dr]\ar@{.}[rr]&&Z\ar[dr]&\\
&Y\ar[dr]\ar@{.}[rr]\ar[ru]&&V\\
&&W\ar[dr]\ar[ru]&\\
&&&I|}
\end{equation}
\end{lem}

\begin{proof}
First, assume that $A$ is representation-finite. Consider $\widetilde{Q}$ as described in   Proposition \ref{not} (a) or (b).
We only analyze (a), since (b) follows similarly.
If $\widetilde{Q}$ is the quiver  $${\xymatrix  @R=0.3cm  @C=0.6cm {a \ar@(ul,dl)[]_{\al}\ar[r]^{\beta}&x}}$$ with $\alpha^n=0$ for $n\geq 2$ and $\alpha\beta=0$, then by Corollary \ref{not2} we know that there exists an irreducible morphism from $I_x$ to
$\tau I_x$. Consider the configuration of almost split sequences that involves such morphism:

$$
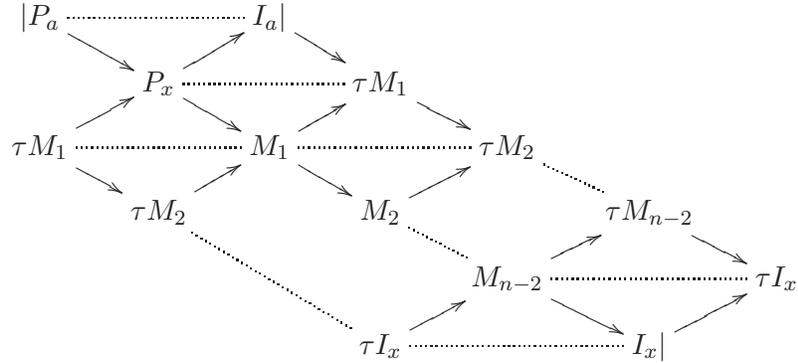
{{\xymatrix  @R=0.3cm  @C=0.6cm {
|P_a\ar@{.}[rr]\ar[dr]&& I_a|\ar[dr]&&&&\\
 & P_x\ar[ur]\ar[dr]\ar@{.}[rr]&&\tau M_1\ar[dr]&&&\\
 \tau M_1\ar[ur]\ar[dr]\ar@{.}[rr] && M_1\ar[ur]\ar[dr]\ar@{.}[rr]&& \tau M_2\ar@{.}[rd]&&\\
 &\tau M_2\ar[ur]\ar@{.}[rrdd]&& M_2\ar[ur]\ar@{.}[rd]&&\tau M_{n-2}\ar[rd]&\\
 &&&&M_{n-2}\ar[ur]\ar[dr]\ar@{.}[rr]&& \tau I_x\\
 &&&\tau I_x\ar[ur]\ar@{.}[rr]&& I_x|\ar[ur]&}}}$$

 \noindent where we identify the modules which are the same. Observe that
if there exist other arrows which start or end in some point of
$\widetilde{Q}$, then the above configuration does not change. To obtain (\ref{eqnvicky}), we conclude that $n=3$. Below, we illustrate the situation.

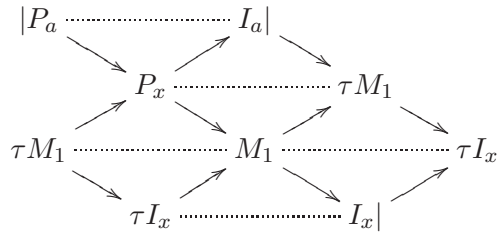
\[
\xymatrix  @R=0.3cm  @C=0.6cm {
|P_a\ar[dr]\ar@{.}[rr]&&I_a|\ar[dr]&&\\
&P_x\ar[dr]\ar@{.}[rr]\ar[ru]&&\tau M_1\ar[dr]&\\
\tau M_1\ar[dr]\ar@{.}[rr]\ar[ru]&&M_1\ar@{.}[rr]\ar[dr]\ar[ru]&&\tau I_x\\
&\tau I_x\ar[ur]\ar@{.}[rr]&&I_x|\ar[ru]&}
\]

\noindent Even though there are cycles of length three, it is not hard to see that
 there are not three irreducible morphisms such that their composition is in $\Re^6\backslash \Re^7$.

Now, if $A$ is representation-infinite, by Proposition \ref{not}, we infer that
$\widetilde{Q}$ is of the form (c). Then $\widetilde{Q}$ is the quiver

$$
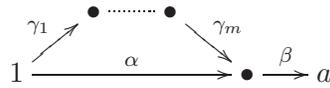
\xymatrix  @R=0.4cm  @C=0.6cm {
&\bullet \ar@{.}[r] & \bullet  \ar[rd]^{\gamma_m}& & 	\\
1 \ar[ru]^{\gamma_1} \ar[rrr]^{\alpha} & & &\bullet \ar[r]^{\beta} &a
}$$

\noindent with $m\geq 1$, $\alpha \beta=0$, and $\gamma_1\cdots\gamma_m\beta\notin I_A$,
and there exists an irreducible morphism from $I_a$ to $\tau I_a$. Let $\Gamma$ be the component
of $\Gamma_A$ such that $I_a\in \Gamma$. Then $\Gamma$ is as follows:

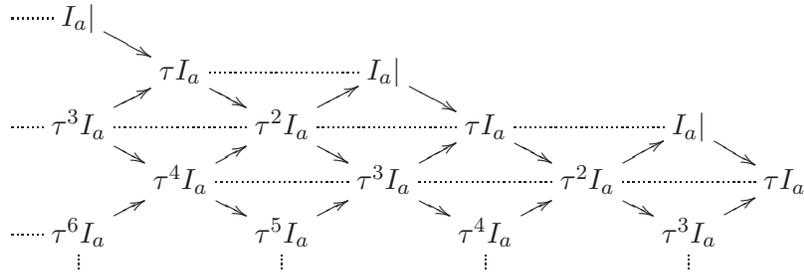
\begin{displaymath}
\xymatrix  @R=0.1cm  @C=0.4cm {
\ar@{.}[r]& I_a|\ar[rd]&&&&&&&\\
&& \tau I_a\ar[rd]\ar@{.}[rr] &&I_a|\ar[rd] &&&&\\
\ar@{.}[r]& \tau^3I_a\ar[rd]\ar[ru]\ar@{.}[rr]&&\tau^2I_a\ar[rd]\ar[ru]\ar@{.}[rr]&& \tau I_a\ar[rd]\ar@{.}[rr]&& I_a|\ar[rd]&\\
&& \tau^4I_a\ar[rd]\ar[ru]\ar@{.}[rr]&& \tau^3I_a\ar[rd]\ar[ru]\ar@{.}[rr]&&\tau^2I_a\ar[rd]\ar[ru]\ar@{.}[rr]&&\tau I_a\\
\ar@{.}[r] &\tau^6I_a\ar[ru]\ar@{.}[d]&&\tau^5I_a\ar[ru]\ar@{.}[d]&& \tau^4I_a\ar[ru]\ar@{.}[d]&& \tau^3I_a\ar@{.}[d]\ar[ru]&\\
&&&&&&&&}
\end{displaymath}

\noindent Observe that
if there exist other arrows which start or end in some point of
$\widetilde{Q}$, the quiver $\Gamma$ does not change. In this case, we do not have a configuration as (\ref{eqnvicky}) since $\tau^4 I_a$ is not an injective module. Therefore, we dismiss this case.
\end{proof}

Now, we are in position to prove the main result of this paper.

\begin{thm} Let $A$ be a string algebra. There are not irreducible morphisms  $f_1:X \rightarrow Y,\, f_2:Y \rightarrow W,$ and $f_3:W \rightarrow V$ between   indecomposable $A$-modules such that $f_3f_2f_1\in \Re^6(X,V)\backslash \Re^7(X,V)$ with $f_2f_1\notin \Re^3(X,W)$ and  $f_3f_2\notin \Re^3(Y,V)$.
\end{thm}

\begin{proof}

Let $f_1:X \rightarrow Y,\, f_2:Y \rightarrow W,$ and $f_3:W \rightarrow V$ be irreducible morphisms as in the statement. By \cite[Theorem 2.2]{CCT}
there is a configuration of almost split sequences as follows:

\begin{equation} \label{tres-morf}
\xymatrix  @R=0.3cm  @C=0.6cm {
	X\ar[dr]_{h_1}\ar@{.}[rr]&& Z\ar[dr]&&\\
&Y\ar[dr]_{h_2}\ar[ur]_{h_4}\ar@{.}[rr]&& V&&\\
&& W\ar[ur]_{h_3}&&&&}
\end{equation}

\noindent such that $h_3h_2h_1=0$, $\alpha'(X)=1$ and $\alpha'(Y)=2$ or its dual.

As we proved in Proposition \ref{cycle3}, there exists a path of irreducible morphisms
between indecomposable modules as follows:

$$\xymatrix  @R=0.3cm  @C=0.6cm {
\psi:X\ar[r]^{g_1}&Y\ar[r]^{g_2}&W\ar[r]^{g_3}&A_3\ar[r]^{g_4}&A_4\ar[r]^{g_5}&A_5\ar[r]^{g_6}&V
}$$

\noindent where $A_3\not \simeq V$ and $A_3 \not \simeq  X$, otherwise, $\psi \in \Re^7(X,V)$.
Moreover, there is cycle of length three $Y\rightsquigarrow Y$ or $W\rightsquigarrow W$, if $A_5\simeq Z$ or $A_5 \simeq W$, respectively.
We claim that the first cycle is not possible in our situation. In fact, assume
that $A_5\simeq Z$, then $A_4\simeq Y$.
Following \cite[Theorem 7]{BS}, the path  $Y\fle W\fle A_3\fle Y \rightarrow W$  is not sectional. Thus,
\begin{enumerate}
\item $Y\simeq \tau A_3$, or
\item $W\simeq \tau Y$, or
\item $A_3\simeq \tau W.$
\end{enumerate}

If $Y\simeq \tau A_3$ then $A_3\simeq V$ contradicting that $\psi\notin \Re^7(X,V)$.

If $W\simeq \tau Y$, then there is a configuration as follows:

\[
\xymatrix  @R=0.3cm  @C=0.6cm {
	X\ar[dr]_{h_1}\ar@{.}[rr]&& Z\ar[dr]&&\\
&Y\ar[dr]_{h_2}\ar[ur]\ar@{.}[rr]&& V\ar[dr]&&\\
&& W\ar[ur]\ar[dr]_{h_3}\ar@{.}[rr]&&Y&&\\
&&&A_3\ar[ur]_{h_4}&&&}
\]

\noindent with  $h_4h_3h_2h_1=0$. Again, the dimension over $k$ of the irreducible morphisms involved is one.
Then $g_i=\alpha_ih_i+\mu_i$, with  $\alpha_i\in k^{*}$ and $\mu_i\in \Re^2$ para $i=1, 2, 3, 4$. Hence
$g_4g_3g_2g_1 \in \Re^5(X,Y)$ a contradiction to the fact that $\psi\notin \Re^7(X,V).$
Therefore, $W\not\simeq \tau Y.$

Finally, suppose that $A_3\simeq \tau W$. Then $\alpha'(A_3)=2$, otherwise, $\alpha'(A_3)=1$, and there is a configuration of almost split sequences as follows:

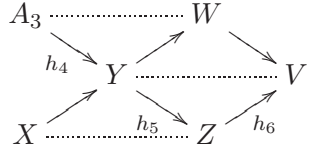
\[
\xymatrix  @R=0.3cm  @C=0.6cm {
	A_3\ar[dr]_{h_4}\ar@{.}[rr]&& W\ar[dr]&&\\
&Y\ar[dr]_{h_5}\ar[ur]\ar@{.}[rr]&& V&&\\
X\ar@{.}[rr]\ar[ur]&& Z\ar[ur]_{h_6}&&&&}
\]

\noindent where $h_6h_5h_4=0$. Since any irreducible morphisms $g_i$ between the involved modules is of the form  $g_i=\alpha_ih_i+\mu_i$, with $\alpha_i\in k^{*}$ and $\mu_i\in \Re^2$ for  $i=4, 5, 6$
then $g_6g_5g_4 \in \Re^4(A_3,V)$, getting that $\psi\in \Re^7(X,V)$, a contradiction. Thus $\alpha'(A_3)=2.$

By Lemma \ref{lema W no iny 3}, $W$ is not an injective module, then there is a configuration of almost split sequences as follows

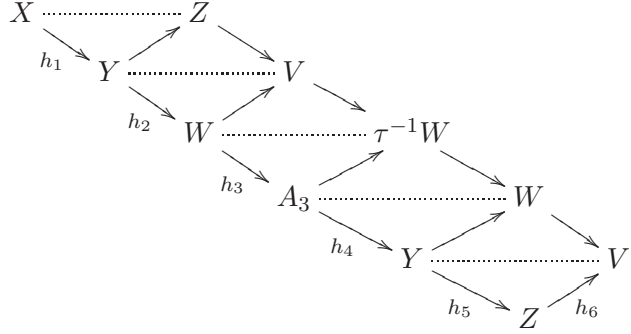
\[
\xymatrix  @R=0.3cm  @C=0.6cm {
	X\ar[dr]_{h_1}\ar@{.}[rr]&& Z\ar[dr]&&\\
&Y\ar[dr]_{h_2}\ar[ur]\ar@{.}[rr]&& V\ar[dr]&&\\
&& W\ar[ur]\ar[dr]_{h_3}\ar@{.}[rr]&&\tau^{-1}W\ar[dr]&&\\
&&&A_3\ar[dr]_{h_4}\ar[ur]\ar@{.}[rr]&&W\ar[dr]&\\
&&&&Y\ar[dr]_{h_5}\ar[ur]\ar@{.}[rr]&& V\\
&&&&&Z\ar[ur]_{h_6}}
\]

Since $g_i=\alpha_ih_i+\mu_i$, with $\alpha_i\in k^{*}$ and  $\mu_i\in \Re^2$ for $i=1, \dots, 6$, then
$\psi\in \Re^7(X,V),$ a contradiction.  Therefore $A_5\not\simeq Z$.

In consequence, $A_5\simeq W$ and the path $\psi$ is as follows:

$$\xymatrix  @R=0.3cm  @C=0.6cm {
\psi:X\ar[r]^{g_1}&Y\ar[r]^{g_2}&W\ar[r]^{g_3}&A_3\ar[r]^{g_4}&A_4\ar[r]^{g_5}&W\ar[r]^{g_6}&V.
}$$

\noindent Observe, that there is a cycle $\varphi:W\fle W$ of length three.

With a similar analysis as before, it is not hard to see that  $A_3\not\simeq \tau W$ and $W\not\simeq \tau A_4$. Then $A_4\simeq \tau A_3$. From Lemmas \ref{W-injectivo 2}, \ref{lema W no iny 2} and \ref{Vicky} we have that $W$ and $A_3$ are not injective. Moreover, if $A_3$ is not injective then we get a contradiction to Lemma \ref{W-A3-notinjective}. Analyzing all the cases we get that  $W\not\simeq A_5$, proving the result.
\end{proof}

\section{On the composition of $n$ irreducible morphisms in $\Re^{n+1}$ which does not belong to the infinite radical}
In this section, we show families of algebras, having $n$ irreducible morphisms such that their compositions belong to $\Re^{n+t}\backslash \Re^{n+t+1}$, with $n\geq 3$ and $ t \geq 4$, and moreover, with the condition that the composition of $n-1$ of them is not in $\Re^n$.

We denote by $(U(m,n-1),I)$ the string algebras whose quiver is

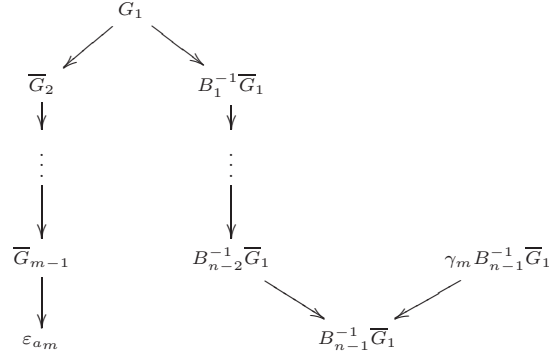
{\begin{displaymath}
    \xymatrix   @R=.3cm  @C=.6cm {
     & a_2\ar[r]^{\ga_2}& \dots\ar[r]^{\ga_{m-1}}& a_m\ar[rd]^{\ga_m}&&\\
     1\ar[ru]^{\ga_1}\ar[rd]_{\be_1} &&&& x& \\
      & b_2\ar[r]_{\be_2}& \dots\ar[r]_{\be_{n-2}}& b_{n-1}\ar[ru]_{\be_{n-1}}&&}
\end{displaymath}}

\noindent with $I=<\ga_{m-1}\ga_{m}>$, for  $m, n\geq 2$.

We shall prove that in the module category of such algebras there are  $n$  irreducible morphisms with composition in $\Re^{n+2m}\backslash\Re^{n+2m+1}$.

\begin{rem}\label{notaBG}
\emph{We define the following strings in $U(m,n-1)$:
\benu \item  $G_j=\ga_1\dots \ga_j,$ and
 $\overline{G}_j=\ga_j\dots \ga_{m-1}$ for  $1\leq j\leq m-1$.
\item $B_i=\be_1\dots \be_i,$ and $\overline{B}_i=\be_i\dots \be_{n-1}$ for $1\leq i\leq n-1$.
\enu
Note that $G_{m-1}=\overline{G}_1$ and $B_{n-1}=\overline{B}_1$.}
\end{rem}

To prove the results of this section, we recall the following notation introduced
in \cite{CG}.
\vspace{.05in}

Let $A$ be a string algebra and let $I=M(D_1D_2)$ be an indecomposable injective $A$-module, where
$D_1=\gamma_s\dots \gamma_1$ is a direct string that starts on a peak,
$D_2=\beta_1^{-1}\dots \beta_r^{-1}$ is an inverse string that ends on a peak. We consider
the following set of strings:

$$\mathcal{C}_{D_2}=\{D D_2 \text{ \emph{ where either }}\; D \text{ \emph{is trivial or}}\; D=D'\ga_1 \text{ \emph{ with }}\; D'  \text{ \emph{ a string}}\}.$$

In a similar way, we can define $\mathcal{C}_{D_1}$ considering $I=M(D_{2}^{-1}D_{1}^{-1})$.
\vspace{.05in}

Next, we recall the quiver $Q_u^e$ defined in \cite{CG}, whose vertices are the strings involved in the sets $\mathcal{C}_{D_1}$
and $\mathcal{C}_{D_2}$.
\vspace{.05in}

Let $A \simeq kQ/I$ and consider the injective $I(u)$, with $u \in Q_0$. Then
\begin{enumerate}
	\item The vertices of $(Q_u^e)_0$ are the strings $C$ in $Q$ such that $e(C)=u,$ where $C$ is either the trivial walk
	$\eps_u$ or $C= C'\al$, with $\al \in Q_1$.
	\item If $a=C$ and $b=C'$ are two vertices of $(Q_u^e)_0$, then there is an arrow from
	$a\fle b$ in $Q_u^e$ if $C'$ is the reduced walk of $\be^{-1}C$, for some $\beta \in Q_1$.
\end{enumerate}

Dually, we can consider an indecomposable projective $A$-module, define the set of strings
and the quiver $Q_u^s$, see \cite{CG}.
\vspace{.05in}

The following results state below  are essential to prove Theorem \ref{kpar}.

\begin{lem}\label{grcomp}
Let $A$ be the algebra $(U(m,n-1),I)$, with
$m,n\geq 2$. Consider the irreducible morphisms
$\iota_{a_m}: \emph{rad}\, P_{a_m} \rightarrow P_{a_m}$ and $\theta_{a_m}:I_{a_m} \rightarrow I_{a_m}/\emph{soc}\,I_{a_m}$,
where $P_{a_m}$ and $I_{a_m}$ are the projective and injective $A$-modules corresponding to the vertex $a_m$, respectively. Then
$d_r(\iota_{a_m})=m+n-1$. Moreover, $d_r(\iota_{a_m})=d_l(\theta_{a_m}).$
\end{lem}

\begin{proof}
Consider $(U(m,n-1),I)$, with $m,n\geq 2$, and the irreducible morphisms $\iota_{a_m}$ and  $\theta_{a_m}$.
By \cite[Proposition 3.2]{CG}, we know that $d_l(\theta_{a_m})$ and $d_r(\iota_{a_m})$ can be compute by the number of vertices of the quivers $Q_{a_m}^e$ and $Q_{a_m}^s$, respectively.

Recall that the vertices of the quiver $Q_{a_m}^e$ are the strings $C$ such that $e(C)={a_m}$, and  $C=\eps_{a_m}$ or $C$ is of the form $C=C'\ga_{m-1}$, with $C'$ a string.
With the notation of Remark \ref{notaBG}, the quiver $Q_{a_m}^e$ is the following:

{\tiny\begin{displaymath}
    \xymatrix   @R=.5cm  @C=.4cm {
    &\overline{G}_1\ar[rd]\ar[ld]&&&\\
    \overline{G}_2\ar[d]&&B_1^{-1}\overline{G}_1\ar[d]&&\\
     \vdots\ar[d]&&\vdots\ar[d]&&\\
      \overline{G}_{m-1}\ar[d]&&B_{n-2}^{-1}\overline{G}_1\ar[dr]&&\ga_mB_{n-1}^{-1}\overline{G}_1\ar[dl]\\
      \varepsilon_{a_m} &&&B_{n-1}^{-1}\overline{G}_1&}
\end{displaymath}}

The cardinal of $(Q_{a_m}^e)_0$  is $m+n$. By \cite[Proposition 3.2]{CG},  $d_l(\theta_{a_m})=\mbox{card}((Q_{a_m}^e)_0)-1$. Hence $d_l(\theta_{a_m})=m+n-1$.

Dually, the quiver $Q_{a_m}^s$ is the following:

{\tiny
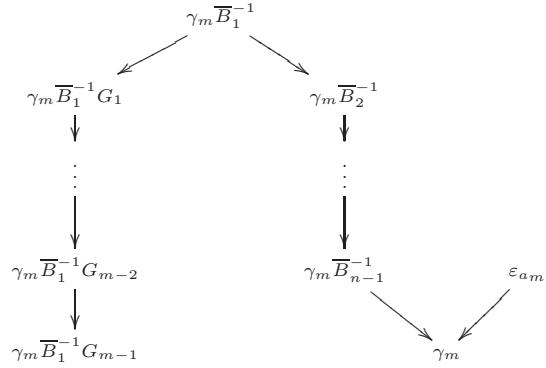
\begin{displaymath}
    \xymatrix   @R=.5cm  @C=.4cm {
    &\ga_m\overline{B}^{-1}_1\ar[rd]\ar[ld]&&&\\
    \ga_m\overline{B}^{-1}_1G_1\ar[d]&&\ga_m\overline{B}^{-1}_2\ar[d]&&\\
     \vdots\ar[d]&&\vdots\ar[d]&&\\
      \ga_m\overline{B}^{-1}_1G_{m-2}\ar[d]&&\ga_m\overline{B}^{-1}_{n-1}\ar[dr]&&\eps_{a_m}\ar[dl]\\
      \ga_m\overline{B}^{-1}_1G_{m-1}&&&\ga_m&}
\end{displaymath}}

\noindent where $Q_{a_m}^s$  has $m+n$ vertices. Therefore, by \cite[Proposition 3.2]{CG} we have that $d_r(\iota_{a_m})=m+n-1$, proving the result.
\end{proof}

\begin{rem}
\emph{Observe that $\ga_m\overline{B}^{-1}_1G_{m-1}=\ga_mB_{n-1}^{-1}\overline{G}_1$ is a vertex in both
quivers $Q_{a_m}^e$ and $Q_{a_m}^s$.
We denote by  $L$
the $A$-module whose string is the mentioned one.}
\end{rem}

Given $X,Y$ and $Z$ indecomposable modules, we denote by
$X\rightsquigarrow Y\rightsquigarrow Z$ a path of irreducible morphisms between indecomposable modules from $X$ to $Z$, going through $Y$.

\begin{prop}\label{seccPI}
Let $A=(U(m,n-1),I)$, with
$m,n\geq 2$, and $P_{a_m}$, $S_{a_m}$ and $I_{a_m}$ be the projective, simple, and  injective module corresponding to the vertex $a_m$, respectively.
Let $L$ be the string module $M(\ga_m\overline{B}^{-1}_1G_{m-1})$.
Then, there is a sectional path $P_{a_m}\rightsquigarrow L\rightsquigarrow S_{a_m}\rightsquigarrow L\rightsquigarrow I_{a_m}$ in $\emph{mod}\,A$.
\noindent Moreover, the cycle $L\rightsquigarrow S_{a_m}\rightsquigarrow L$
has length $2m$.
\end{prop}

\begin{proof}
Consider the irreducible morphism $\theta_{a_m}:I_{a_m}  \rightarrow I_{a_m}/\mbox{soc}I_{a_m}$.
The module $I_{a_m}/\mbox{soc}I_{a_m}$ is indecomposable. Moreover, $\mbox{Ker}(\theta_{a_m})=S_{a_m}$
and by Lemma \ref{grcomp},  $d_l(\theta_{a_m})=m+n-1$. By \cite[Proposition 2.5]{CG}, there is a configuration of almost split sequences as follows:

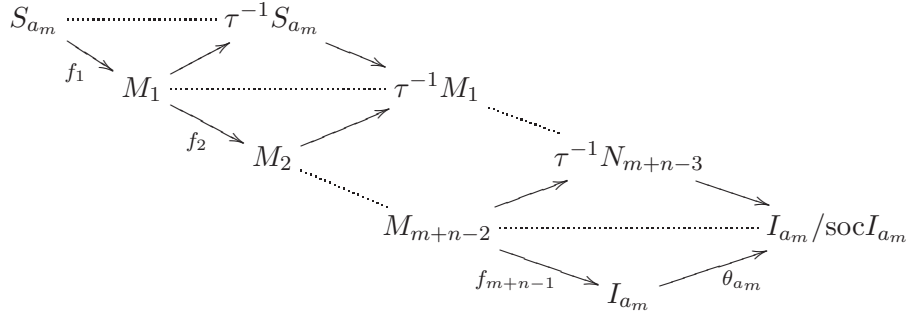
\begin{displaymath}
\xymatrix  @R=0.3cm  @C=0.6cm {
	S_{a_m}\ar[dr]_{f_1} \ar@{.}[rr]& & \tau^{-1}S_{a_m}\ar[dr]& & &\\
	& M_1 \ar[ur]\ar[dr]_{f_2} \ar@{.}[rr]& & \tau^{-1}M_1 \ar@{.}[dr] & &   \\
	&  & M_2\ar[ur]\ar@{.}[dr] & & \tau^{-1}N_{m+n-3}\ar[dr] &  \\
	& & &   M_{m+n-2}\ar@{.}[rr]\ar[dr]_{f_{m+n-1}}\ar[ur] & & I_{a_m}/\mbox{soc}I_{a_m}  \\
	& &  & &  I_{a_m}\ar[ur]_{\theta_{a_m}} &}
\end{displaymath}

\noindent where the path $S_{a_m}\fle M_1\fle \dots \fle M_{m+l-2}\fle I_{a_m}$
is sectional.

On the other hand, the modules of such a path are in correspondence with the string modules $M(C)$, where $C$ are vertices of $Q_{a_m}^e$.
In particular, $L=M(\ga_m\overline{B}^{-1}_1G_{m-1})$, is a module that appears in such a path. Moreover, $L\neq S_{a_m}$ and  $L\neq I_{a_m}$.
Hence,  the path is of the form  $S_{a_m}\rightsquigarrow L\rightsquigarrow I_{a_m}$.

We claim that the length of the path $S_{a_m}\rightsquigarrow L$ is $m$.  To prove our claim, we order the strings of the set $\mathcal{C}_{\eps_{a_m}}$ as follows;   $C_i<C_{i+1}$ if there is an
irreducible morphism $M(C_i)\fle M(C_{i+1}).$  To determine such order in the strings, we may analyze if the strings start on a peak.

Let $C_0=\eps_{a_m}^{-1}$. Since $C_0$  does not start  on a peak,
we define $C_1=\ga_{m-1}\eps_{a_m}^{-1}=\overline{G}_{m-1}$.
Then there is an irreducible morphism $S_{a_m}=M(C_0)\fle M(C_1)$.

Observe that for $2\leq j\leq m-1$,
the strings $\overline{G}_j=\ga_j\dots \ga_{m-1}$
do not start on a peak.
Moreover, following \cite{BR}, we observe that there exist irreducible morphisms
$M(\overline{G_2})\rightarrow M(B_{n-1}^{-1}\overline{G_1})$ and
$M(\overline{G_j})\rightarrow M(\overline{G_{j-1}})$,
for $3\leq j \leq m-1$. Continuing with the order in the set $\mathcal{C}_{\eps_{a_m}}$,
for  $2\leq i\leq m-2$, we define the strings $C_i=\overline{G}_{m-i}$
and  $C_{m-1}=B_{n-1}^{-1}\overline{G}_1$.

Finally,  $C_{m-1}$ does not start on a peak. Then $C_{m}=\ga_mB_{n-1}^{-1}\overline{G}_1$ and
$M(C_m) =L$.
Hence, we have a path of  irreducible morphisms as follows

\begin{equation}\label{SmIm}
S_{a_m}=M(C_0)\fle M(C_1)\fle \dots \fle M(C_{m-1})\fle M(C_m)=L\,\rightsquigarrow \,I_{a_m}
\end{equation}

\noindent where the path $S_{a_m}\rightsquigarrow L$ has length $m$.

Dually, if we consider the irreducible morphism $\iota_{a_m}:\mbox{rad}\, P_{a_m}\fle P_{a_m}$ then $\mbox{rad}\,P_{a_m}$ is indecomposable and $\mbox{Coker}(\iota_{a_m})=S_{a_m}$.
By \cite[Proposition 2.5]{CG} and Lemma \ref{grcomp},
there is a sectional path $P_{a_m}\rightsquigarrow S_{a_m}$ of length $m+n-1$. Again, the modules of such a path are in correspondence with the vertices of $Q_{a_m}^s$. In particular,
$L=M(\ga_m\overline{B}^{-1}_1G_{m-1})$ is a module of such a path. Moreover, $L\neq P_{a_m}$ and $L\neq S_{a_m}$.
Hence, we have a path of the form $P_{a_m}\rightsquigarrow L\rightsquigarrow S_{a_m}.$

Again we can prove that $L\rightsquigarrow S_{a_m}$ has length $m$,  by considering an order on the strings of the set $\mathcal{D}_{\eps_{a_m}}$ as follows, $D_i<D_{i+1}$ if there is an irreducible morphism $M(D_{i+1})\fle M(D_i)$.
In this case, to order the strings, we have to analyze if the strings ends in a deep.

Similarly, we can prove that $M(D_m)=L$ and that there is a path of irreducible morphisms of the form:

\begin{equation}\label{PmSm}
P_{a_m}\rightsquigarrow L=M(D_m)\fle M(D_{m-1})\fle \dots \fle M(D_1)\fle M(D_0)=S_{a_m}
\end{equation}

\noindent where the path $L\rightsquigarrow S_{a_m}$  has length $m$.

Now, from the paths (\ref{SmIm}) and (\ref{PmSm}) we obtain the path
\begin{equation}\label{PmIm}
P_{a_m}\rightsquigarrow L\rightsquigarrow S_{a_m}\rightsquigarrow L\rightsquigarrow I_{a_m},
\end{equation}
\noindent where the cycle $L\rightsquigarrow S_m\rightsquigarrow L$ clearly has length $2m$.

It is left to prove that the path (\ref{PmIm}) is sectional.
By construction the paths
$P_{a_m}\rightsquigarrow L\rightsquigarrow S_{a_m}$ and $S_{a_m}\rightsquigarrow L\rightsquigarrow I_{a_m}$
are sectional. We must analyze the path $M(D_1)\fle S_{a_m} \fle M(C_1)$.
Note that $M(D_1)$ is injective,
since $D_1=\ga_mB_{n-1}^{-1}$, where $B_{n-1}$ and $\ga_m$  are string ending in a peak.
More precisely, since $e(B_{n-1})=x=e(\ga_m),$
then $M(D_1)=I_x$.
Therefore, $M(C_1)\not\simeq \tau^{-1}I_x$ and the path (\ref{PmIm}) is sectional, proving the result.
\end{proof}

\begin{prop}\label{fLN}
Let $A=(U(m,n-1),I)$, with
$m,n\geq 2$. Consider  $L=M(\ga_m{B}^{-1}_{n-1}\overline{G}_{1})$
and $N=M({B}^{-1}_{n-2}\overline{G}_{1})$. Then $f:L\fle N$ is an irreducible epimorphism with  $d_l(f)=n-1$.
\end{prop}

\begin{proof}
Consider  $L=M(C)$ and  $N=M(D)$,  where  $C=\ga_{m}\be_{n-1}^{-1}\dots\be_1^{-1}\ga_1\dots \ga_{m-1}$
and $D=\be_{n-2}^{-1}\dots\be_1^{-1}\ga_{1}\dots\ga_{m-1} $ is a string not starting in a deep.
%Moreover, $D_c=\ga_m\be_{n-1}^{-1}D=C$.
By \cite{BR} we have that $f:M(C)\fle M(D)$  is an irreducible epimorphism,
where $\mbox{Ker}(f)\simeq M(C)/M(D)\simeq  M(\ga_m)=P_{a_m}.$
Since $A$ is representation finite, then $d_l(f)< \infty$.

Now, we compute the left degree of $f$.  Since $e(\ga_m)=x$,
we consider the module $I_x=M(B_{n-1}\ga_m^{-1})$.
The indecomposable direct summands of $I_x/\mbox{soc}I_x$ are $J_1(x)=M(B_{n-2})$ and $J_2(x)=S_{a_m}$.

Consider the irreducible morphism $h:I_x\fle J_1(x)$,
where $\mbox{Ker}(f)\simeq M(\ga_m)\simeq \mbox{Ker}(h)$.
Assume that $d_l(h)=l$. Then $f:L\fle N$ is one of the morphisms $g_i:X_i\fle \tau^{-1}X_{i-1}$ of the following configuration of almost split sequences:

\begin{displaymath}
\xymatrix  @R=0.3cm  @C=0.6cm {
	P_{x}\ar[dr]_{f_1} \ar@{.}[rr]& & \tau^{-1}P_x\ar[dr]& & &\\
	& X_1 \ar[ur]_{g_1}\ar[dr]_{f_2} \ar@{.}[rr]& & \tau^{-1}X_1 \ar@{.}[dr] & &   \\
	&  & X_2\ar[ur]_{g_2}\ar@{.}[dr] & & \tau^{-1}X_{l-2}\ar[dr] &  \\
	& & &   X_{l-1}\ar@{.}[rr]\ar[dr]_{f_{l}}\ar[ur]_{g_{l-1}} & & J_1(x)  \\
	& &  & &  I_x\ar[ur]_h &}
\end{displaymath}

\noindent where  $\al'(P_x)=1$ and  $\phi: P_x \rightarrow X_1 \rightarrow \dots \rightarrow X_{l-1} \rightarrow I_x$
is a sectional path. The modules that appear in $\phi$ are the string modules of the set $C_{\gamma_m}$.
In particular, $L$ is one of such modules. Since $L\not \simeq P_x$ and  $L \not \simeq  I_x$, then $L\simeq X_j$, for some $j$, $1\leq j\leq l-1$.

On the other hand, by the proof of Proposition \ref{seccPI}, there is a sectional path

$$\rho: P_{a_m}\fle M_1\fle \dots \fle L\fle \dots \fle M_{m+n-3}\fle I_x\fle S_{a_m}$$

\noindent of length $m+n-1$ and where $L \rightsquigarrow S_{a_m}$
has length $m$.

Since $\mbox{dim}_k(\mbox{Hom}_A(P_{a_m},S_{a_m}))=1$, then $l=m+n-2$. We claim that for each $i$, $1\leq i\leq l-1$ we have that $M_i\simeq X_i$.
In fact, since $\al'(P_m)=1$, then $X_1\simeq M_1$. Now, since $\rho$ is a sectional path, $M_2\not \simeq \tau^{-1}P_x$.
Then $X_2\simeq M_2$. Following this argument, we get that $X_i\simeq M_i$, for $1\leq i\leq l-1$. Since the path $L\rightsquigarrow S_{a_m}$
has length $m$, then $L\simeq X_{n-1}$ and therefore we obtain that
$d_l(f)=n-1$.
 \end{proof}

\begin{rem}\label{obsN}
\emph{By the proofs of Proposition \ref{seccPI} and \ref{fLN}, we have the existence of a sectional path}

$$P_{a_m}\stackrel{\phi} \rightsquigarrow L \rightsquigarrow S_{a_m} \rightsquigarrow L\rightsquigarrow I_{a_m}$$

\noindent \emph{where the path $\phi$ is of length $n-1$ and the cycle $L\rightsquigarrow L$ is of length $2m$.
Moreover, we know that there exists an irreducible morphism $f: L\fle N$. We claim that the module $N$ belongs to such
sectional path. In fact, in the proof of Proposition \ref{seccPI}, we give an order for the strings
$C_1, \dots, C_m$ of the set $\mathcal{C}_{\eps_{a_m}}$, where
$C_m=\ga_mB_{n-1}^{-1}\overline{G}_1$ and  $M(C_m)=L$. }

\emph{Observe that $C_m$ is a string that starts on a peak. Hence $C_m =\ga_m\be^{-1}_{n-1}B_{n-2}^{-1}\overline{G}_1$.
Then $C_{m+1}=B_{n-2}^{-1}\overline{G}_1$ and  $N=M(C_{m+1})$. Moreover, $N$
does not belong to the path $P_{a_m}\rightsquigarrow S_{a_m}$, because the string
$B_{n-2}^{-1}\overline{G}_1$ is not a vertex of the quiver $Q^s_{a_m}$.
Hence, we conclude that the above sectional path is of the form}

\begin{equation*}
P_{a_m} \rightsquigarrow L \rightsquigarrow  S_{a_m} \rightsquigarrow L\fle N\rightsquigarrow I_{a_m}
\end{equation*}

\noindent \emph{where the arrow denotes an irreducible morphism.}
\end{rem}

Now, we are in position to prove the theorem.

\begin{thm}\label{kpar}
Let $A =(U(m,n-1),I)$, with
$m,m\geq 2$. Then there are irreducible morphisms
$h_i:X_i\fle X_{i+1}$ for  $1\leq i\leq n$, between indecomposable $A-$modules,
such that
$h_n\dots h_1\in \Re^{n+2m}(X_1,X_{n+1})\backslash \Re^{n+2m+1}(X_1,X_{n+1})$,
$h_{n-1}\dots h_1\notin \Re^{n}(X_1,X_{n})$ and $h_n\dots h_2\notin \Re^{n}(X_2,X_{n+1})$.
\end{thm}

\begin{proof}
Consider the irreducible epimorphism $f:L\fle N$ from Proposition \ref{fLN}.
Then $d_l(f)=n-1$ and  $\mbox{Ker}(f)=P_m$. Then there is a configuration of almost split sequences as follows:

\begin{displaymath}
\xymatrix  @R=0.3cm  @C=0.6cm {
	P_m\ar[dr]_{f_1}\ar@{.}[rr]& & \tau^{-1}P_m\ar[dr]& & &\\
	& Y_1 \ar[ur]\ar[dr]_{f_2} \ar@{.}[rr]& & \tau^{-1}Y_1 \ar@{.}[dr] & &   \\
	&  & Y_2\ar[ur]\ar@{.}[dr]  & & \tau^{-1}Y_{n-3}\ar[dr] &  \\
	& & &   Y_{n-2}\ar@{.}[rr]\ar[dr]_{f_{n-1}}\ar[ur] & & N  \\
	& &  & & L\ar[ur]_{f} &
}
\end{displaymath}
\noindent where
$\delta:P_m\fle Y_1\fle \dots \fle Y_{n-2}\fle L$
is a sectional path of length $n-1$ and  $ff_{n-1}\dots f_1=0$. By Remark \ref{obsN}
there is a sectional path

$$P_m\stackrel{\phi}\rightsquigarrow L\stackrel{\rho_1}\rightsquigarrow S_m \stackrel{\rho_2}\rightsquigarrow L\fle N \rightsquigarrow I_m$$

\noindent where $\ell(\phi)=n-1$ and $\ell(\rho_2\rho_1)=2m$.
Moreover, the modules in the path  $\phi:P_m \rightsquigarrow L$ are the same that the ones in the path $\delta$.

Consider $X_1=P_m$, $X_n=L$, $X_{n+1}=N$
and $X_i=Y_{i+1}$ for $1\leq i\leq n-1$. We define the irreducible morphisms $h_i=f_i$ for $1\leq i\leq n-2$,
$h_{n-1}=f_{n-1}+f_{n-1}\rho$, where $\rho:L\rightsquigarrow L$
is a composition of $2m$  irreducible morphisms which form part of the sectional path $\rho_2\rho_1$, and $h_n=f$.
Then the composition

\begin{equation*}
\begin{array}{lll}
h_n\dots h_1&=&f(f_{n-1}+f_{n-1}\rho)f_{n-2}\dots f_{1}\\
&=&ff_{n-1}f_{n-2}\dots f_{1}+ff_{n-1}\rho f_{n-2}\dots f_{1}\\
&=&ff_{n-1}\rho f_{n-2}\dots f_{1}.
\end{array}
\end{equation*}

\noindent belongs to $ \Re^{n+2m}(X_1,X_{n+1})\backslash\Re^{n+2m+1}(X_1,X_{n+1})$,
because the morphisms belong to a sectional path of length $n+2m$.
Furthermore, $h_{n-1}\dots h_1\notin \Re^{n}(X_1,X_{n})$  and by  \cite[Proposition 2.3]{Cha2}
we have that $h_n\dots h_2\notin \Re^{n}(X_2,X_{n+1})$, proving the result.
\end{proof}

In the families of algebras presented in Theorem \ref{kpar}, there are $n$ irreducible morphisms such that their composition belong to  $\Re^{n+t}\backslash\Re^{n+t+1}$,
for  $t\geq 4$ and moreover where $t$ is an even number.

Below, we present a family of algebras for $t$ an odd number.
Consider $(V(m,n),J)$ for $n\geq 3$ and  $m\geq 2$ as follows:

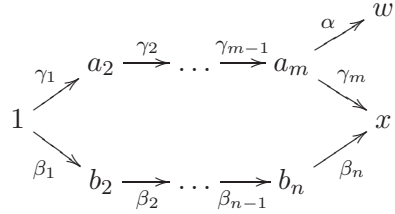
{\begin{displaymath}
    \xymatrix   @R=.3cm  @C=.6cm {
    &&&&w&&\\
     & a_2\ar[r]^{\ga_2}& \dots\ar[r]^{\ga_{m-1}}& a_m\ar[rd]^{\ga_m}\ar[ru]^{\al}&&\\
     1\ar[ru]^{\ga_1}\ar[rd]_{\be_1} &&&& x& \\
      & b_2\ar[r]_{\be_2}& \dots\ar[r]_{\be_{n-1}}& b_{n}\ar[ru]_{\be_{n}}&&}
\end{displaymath}}

\noindent with $J=<\ga_{m-1}\ga_{m}>$.

We only state the result, since it can be proved with similar techniques as in Theorem \ref{kpar}.

\begin{thm}\label{kimpar}
Let $A =(V(m,n-2),J)$, with
$m\geq 2$ and $n\geq 3$. Then there are irreducible morphisms
$h_i:X_i\fle X_{i+1}$ for $1\leq i\leq n$, between indecomposable $A$-modules,
 such that
$h_n\dots h_1\in \Re^{n+2m+1}(X_1,X_{n+1})\backslash \Re^{n+2m+2}(X_1,X_{n+1})$,
$h_{n-1}\dots h_1\notin \Re^{n}(X_1,X_{n})$ and $h_n\dots h_2\notin \Re^{n}(X_2,X_{n+1})$.
\end{thm}

\end{document}